\documentclass{amsart}
\usepackage{amssymb,amsmath,graphicx}
\usepackage{xcolor}
\usepackage{enumerate}
\newtoks\prt 
 \newtheorem{thma}{Theorem}
 
\newtheorem{thm}{Theorem}[section]

\newtheorem{ques}[thm]{Question} 
\newtheorem{lemma}[thm]{Lemma} 
\newtheorem{prop}[thm]{Proposition} 
\newtheorem*{prop-without-numbering}{Proposition} 
\newtheorem{cor}[thm]{Corollary} 
\newtheorem{fact}[thm]{Fact} 

\newtheorem{obs}[thm]{Observation}
\newtheorem{example}[thm]{Example}

\theoremstyle{definition} 
\newtheorem{example2}[thm]{Example} 
\newtheorem{remark}[thm]{Remark}

\def\eqn#1$$#2$${\begin{equation}\label#1#2\end{equation}} 
 \def\J#1#2#3{ \left\{ #1,#2,#3 \right\} }

\def\ce{\mathbb C}

\def\ep{\varepsilon} 
 
\def\en{\mathbb N} 
\def\er{\mathbb R} 
 
\def\O{\mathbb O}

\def\dist{\operatorname{dist}}

\def \f{\boldsymbol{f}} 
 
\def \g{\boldsymbol{g}}

\def \h{\boldsymbol{h}} 
\def \uu{\boldsymbol{u}} 
\def \vv{\boldsymbol{v}}

\def \card {\operatorname{card}}

\def \reg {\partial _{\kern1pt\text{reg}}}

\def\ip#1#2{\left\langle#1,#2\right\rangle}

\def\di{\,\mbox{\rm d}}

\newcommand{\Norm}[1]{\Bigl\|#1\Bigr\|}
\newcommand{\norm}[1]{\left\|#1\right\|}

\newcommand{\betr}[1]{| #1  |}

\newcommand{\tr}[1]{\operatorname{tr}\bigl(#1\bigr)}

\newcommand{\abs}[1]{\left|#1\right|}
\newcommand{\setsep}{;\,}

\title{On optimality of constants in the Little Grothendieck Theorem}

\author[O.F.K. Kalenda]{Ond\v{r}ej F.K. Kalenda}

\author[A.M. Peralta]{Antonio M. Peralta}

\author[H. Pfitzner]{Hermann Pfitzner}

\address{Charles University, Faculty of Mathematics and Physics, Department of
Mathematical Analysis, Sokolovsk{\'a} 86, 186 75 Praha 8, Czech Republic}
\email{kalenda@karlin.mff.cuni.cz}
\address{Unidad de Excelencia María de Maeztu IMAG. Departamento de An{\'a}lisis Matem{\'a}tico, Facultad de
Ciencias, Universidad de Gra\-na\-da, 18071 Granada, Spain.}
\email{aperalta@ugr.es}
\address{Universit\'{e} d'Orl\'{e}ans,
Département de mathématiques, IDP,
BP 6759,
F-45067 Orl\'{e}ans Cedex 2,
France}
\email{pfitzner@univ-orleans.fr}



\begin{document}

\begin{abstract} We explore the optimality of the constants making valid the recently established Little Grothendieck inequality for JB$^*$-triples and JB$^*$-algebras. In our main result we prove that for each bounded linear operator $T$ from a JB$^*$-algebra $B$ into a complex Hilbert space  $H$ and $\varepsilon>0$, there is a norm-one functional $\varphi\in B^*$ such that
$$\norm{Tx}\le(\sqrt{2}+\varepsilon)\norm{T}\norm{x}_\varphi\quad\mbox{ for }x\in B.$$ The constant appearing in this theorem improves the best  
value known up to date (even for C$^*$-algebras).
We also present an easy example witnessing that the constant cannot be strictly smaller than $\sqrt2$, hence our main theorem  
 is `asymptotically optimal'.
For type I JBW$^*$-algebras  we establish a canonical decomposition of normal functionals which may be used to prove the main result in this special case and also seems to be of an independent interest. As a tool we prove a measurable version of the Schmidt representation of compact operators on a Hilbert space. 
\end{abstract}

\maketitle

\section{Introduction}

We investigate the optimal values of the constant in the Little Grothendieck theorem for JB$^*$-algebra. The story begins in 1956 
when Grothendieck \cite{grothendieck1956resume} proved his famous theorem on factorization of bilinear forms on spaces of continuous functions through Hilbert spaces. A weaker form of this result, called Little Grothendieck Theorem, can be formulated as a canonical factorization of bounded linear operators from spaces of continuous functions into a Hilbert space. It was also proved by Grothendieck \cite{grothendieck1956resume} (see also \cite[Theorem 5.2]{pisier2012grothendieck}) and reads as follows.

\begin{thma}\label{T-C(K)}
There is a universal constant $k$ such that 
for any bounded linear operator  $T:C(K)\to H$, where $K$ is a compact space and $H$ is a Hilbert space, there is a Radon probability measure $\mu$ on $K$ such that
$$\norm{Tf}\le k\norm{T} \left(\int \abs{f}^2\di\mu\right)^{\frac12}\quad\mbox{ for }f\in C(K).$$
Moreover, the optimal value of $k$ is $\frac2{\sqrt\pi}$ in the complex case and $\sqrt{\frac\pi2}$ in the real case.
\end{thma}

The Grothendieck theorem was later extended to the case of C$^*$-algebras by Pisier \cite{pisier1978grothendieck} and Haagerup \cite{haagerup1985grothendieck}. Its `little version' reads as follows. Henceforth, all Hilbert spaces considered in this note will be over the complex field. 

\begin{thma}\label{T-C*alg}
Let $A$ be a C$^*$-algebra, $H$ a Hilbert space and $T:A\to H$ a bounded linear operator. Then there are two states $\varphi_1,\varphi_2\in A^*$ such that
$$\norm{Tx}\le\norm{T}\left(\varphi_1(x^*x)+\varphi_2(xx^*)\right)^{\frac12}\quad \mbox{ for }x\in A.$$
Moreover, the constant $1$ on the right-hand side is optimal.
\end{thma}

The positive part of the previous theorem is due to Haagerup \cite{haagerup1985grothendieck}, the optimality result was proved by Haagerup and Itoh in \cite{haagerup-itoh} (see also \cite[Section 11]{pisier2012grothendieck}). Let us recall that a \emph{state} on a C$^*$-algebra is a positive functional of norm one, hence in the case of a complex $C(K)$ space (which is a commutative C$^*$-algebra), a state is just a functional represented by a probability measure. Hence, as a consequence of Theorem~\ref{T-C*alg}
we get a weaker version of the complex version of Theorem~\ref{T-C(K)} with
$k\le\sqrt{2}$.

Let us point out that Theorem~\ref{T-C*alg} is specific for (noncommutative) C$^*$-algebras due to the asymmetric role played there by the products $xx^*$ and $x^*x$.
To formulate its symmetric version recall that the Jordan product on a C$^*$-algebra $A$ is defined by
$$x\circ y=\frac12(xy+yx)\quad\mbox{ for }x,y\in A.$$
Using this notation we may formulate the following consequence of Theorem~\ref{T-C*alg}.

\begin{thma}\label{T:C*alg-sym}
Let $A$ be a C$^*$-algebra, $H$ a Hilbert space and $T:A\to H$ a bounded linear operator. Then there is a state $\varphi\in A^*$ such that
$$\norm{Tx}\le2\norm{T}\varphi(x\circ x^*)^{\frac12}\quad \mbox{ for }x\in A.$$
\end{thma}

To deduce Theorem~\ref{T:C*alg-sym} from Theorem~\ref{T-C*alg} it is enough to take $\varphi=\frac12(\varphi_1+\varphi_2)$ and to use positivity of the elements $xx^*$ and $x^*x$. However, in this case the question on optimality of the constant remains open.

\begin{ques}
Is the constant $2$ in Theorem~\ref{T:C*alg-sym} optimal?
\end{ques}

It is easy to show that the constant should be at least $\sqrt2$ (see Example~\ref{ex:alg vs triple} below) and, to the best of our knowledge, no counterexample is known showing that $\sqrt{2}$ is not enough.

A further generalization of the Grothendieck theorem, to the setting of JB$^*$-triples (see Section~\ref{sec: notation} for basic definitions and properties), was suggested by Barton and Friedman \cite{barton1987grothendieck}. However, their proof contained a gap found later by Peralta and Rodr\'{\i}guez Palacios \cite{peralta2001little,peralta2001grothendieck} who proved a weaker variant of the theorem. A correct proof was recently provided by the authors in \cite{HKPP-BF}. The `little versions' of these results are summarized in the following theorem. 

\begin{thma}\label{T:triples}
Let $E$ be a JB$^*$-triple, $H$ a Hilbert space and $T:E\to H$ a bounded linear operator.
\begin{enumerate}[$(1)$]
    \item If $T^{**}$ attains its norm, there is a norm-one functional $\varphi\in E^*$ such that
    $$\norm{Tx}\le\sqrt{2}\norm{T}\norm{x}_\varphi\quad\mbox{ for }x\in E.$$
    \item Given $\varepsilon>0$, there are norm-one functionals $\varphi_1,\varphi_2\in E^*$ such that
     $$\norm{Tx}\le(\sqrt{2}+\varepsilon)\norm{T}\left(\norm{x}^2_{\varphi_1}+\varepsilon\norm{x}^2_{\varphi_2}\right)^{\frac12}\quad\mbox{ for }x\in E.$$
     \item Given $\varepsilon>0$, there is a norm-one functional $\varphi\in E^*$ such that
    $$\norm{Tx}\le(2+\varepsilon)\norm{T}\norm{x}_\varphi\quad\mbox{ for }x\in E.$$
\end{enumerate}
\end{thma}

The pre-hilbertian seminorms $\norm{\cdot}_\varphi$ used in the statement are defined in Subsection~\ref{subsec:seminorms} below.

Let us comment the history and the differences of the three versions.
It was claimed in \cite[Theorem 1.3]{barton1987grothendieck} that assertion $(1)$ holds without the additional assumption on attaining the norm, because the authors assumed this assumption is satisfied automatically. In \cite{peralta2001little} and \cite[Example 1 and Theorem 3]{peralta2001grothendieck} it was pointed out that this is not the case and assertion $(2)$ was proved using a variational principle from \cite{Poliquin-Zizler}. 
In \cite[Lemma 3]{peralta2001grothendieck} also assertion $(1)$ was formulated.

Note that in $(2)$ not only the constant $\sqrt2$ is replaced by a slightly larger one, but also the pre-hilbertian seminorm on the right-hand side is perturbed. This perturbation was recently avoided in \cite[Theorem 6.2]{HKPP-BF}, at the cost of squaring the constant.
Further, although the proof from \cite{barton1987grothendieck} was not correct, up to now there is no counterexample to the statement itself.
In particular, the following question remains open.

\begin{ques}
What is the optimal constant in assertion $(3)$ of Theorem~\ref{T:triples}? In particular, does assertion $(1)$ of the mentioned theorem hold without assuming the norm-attainment?
\end{ques}

The main result of this note is the following partial answer.

\begin{thm}\label{t constant >sqrt2 in LG for JBstar algebras}
Let $B$ be a JB$^*$-algebra, $H$ a Hilbert space and $T:B\to H$ a bounded linear operator. Given $\varepsilon>0$, there is a norm-one functional $\varphi\in B^*$ such that
    $$\norm{Tx}\le(\sqrt{2}+\varepsilon)\norm{T}\norm{x}_\varphi\quad\mbox{ for }x\in B.$$
In particular, this holds if $B$ is a C$^*$-algebra.    
\end{thm}

Note that JB$^*$-algebras form a subclass of JB$^*$-triples and can be viewed as a generalization of C$^*$-algebras (see the next section). We further remark that the previous theorem is `asymptotically optimal' as the constant cannot be strictly smaller than $\sqrt2$ by Example~\ref{ex:Tx=xxi} below.

The paper is organized as follows. Section~\ref{sec: notation} contains basic background on JB$^*$-triples and JB$^*$-algebras. In Section~\ref{sec:majorizing} we formulate the basic strategy of the proof using majorization results for pre-hilbertian seminorms. 

In Section~\ref{sec:type I} we deal with a large subclass of JBW$^*$-algebras (finite ones and those of type I). The main result of this section is Proposition~\ref{P:type I approx} which provides a canonical decomposition of normal functionals on the just commented JBW$^*$-algebras. 
This statement may be used to prove the main result in this special case and, moreover, it seems to be of an independent interest.
As a tool we further establish a measurable version of Schmidt decomposition of compact operators (see Theorem~\ref{T:measurable Schmidt}).

In Section~\ref{sec:JW*} we address Jordan subalgebras of von Neumann algebras. Section~\ref{sec:proofs} contains the synthesis of the previous sections, the proof of the main result and some consequences. In particular, we show that Theorem~\ref{T-C*alg} (with the precise constant) follows easily from Theorem~\ref{t constant >sqrt2 in LG for JBstar algebras}.

Section~\ref{sec:problems} contains several examples witnessing optimality of some results and related open problems. In Section~\ref{sec:triples} we discuss the possibility of extending our results to general JB$^*$-triples.

\section{Basic facts on JB$^*$-triples and JB$^*$-algebras}\label{sec: notation}

It is known that in most cases, like in $B(H)$, the hermitian part of a C$^*$-algebra $A$ need not be a subalgebra of $A$ because it is not necessarily closed for the associative product. This instability can be avoided, at the cost of loosing associativity, by replacing the associative product $a b$ in $A$ with the \emph{Jordan product} defined by \begin{equation}\label{eq special Jordan product} a\circ b := \frac12 (a b + ba).
\end{equation}

This may be seen as an inspiration for the following abstract definitions.
A real or complex \emph{Jordan algebra} is a non-necessarily associative algebra $B$ over $\mathbb{R}$ or $\mathbb{C}$ whose multiplication (denoted by $\circ$) satisfies the identities:\begin{equation}\label{eq Jordan axioms} x\circ y = y \circ x \hbox{ (commutative law)  and }
    ( x\circ y )\circ  x^2 = x \circ ( y \circ x^2 ) \hbox{ (Jordan identity)}
\end{equation}  for all $x,y\in B$, where $x^2 = x\circ x$.\smallskip

Jordan algebras were the mathematical structures designed by the theoretical physicist P.~Jordan to formalize the notion of an algebra of observables in quantum mechanics in 1933. The term ``Jordan algebra'' was introduced by A.A.~Albert in the  1940s. Promoted by the pioneering works of I.~Kaplanski, E.M.~Alfsen, F.W.~Shultz, H.~Hanche-Olsen, E.~St\"{o}rmer, J.D.M.~Wright and M.A.~Youngson, JB$^*$- and JBW$^*$-algebras are Jordan extensions of C$^*$- and von Neumann algebras. A {\em JB$^*$-algebra} is a complex Jordan algebra $(B,\circ)$ equipped with a complete norm $\|\cdot\|$ and an involution $*$ satisfying the following axioms: \begin{enumerate}[$(a)$]
\item $\norm{x\circ y}\le\norm{x}\norm{y}$ for $x,y\in B$;
\item $\norm{U_{x} (x^*)}=\norm{x}^3$ for $x\in B$ \hfill ({\em a Gelfand-Naimark type axiom}),
\end{enumerate} where $U_{x} (y) = 2(x\circ y)\circ x-x^2\circ y$ ($x,y\in B$). These axioms guarantee that the involution of every JB$^*$-algebra is an isometry (see \cite[Lemma 4]{youngson1978vidav} or \cite[Proposition 3.3.13]{Cabrera-Rodriguez-vol1}).\smallskip

JB$^*$-algebras were also called \emph{Jordan C$^*$-algebras} by I. Kaplansky and other authors at the early stages of the theory.\smallskip

Every C$^*$-algebra is a JB$^*$-algebra with its original norm and involution and the Jordan product defined in \eqref{eq special Jordan product}. Actually, every norm closed self-adjoint Jordan subalgebra of a C$^*$-algebra is a JB$^*$-algebra. Those JB$^*$-algebras obtained as JB$^*$-subalgebras of C$^*$-algebras are called \emph{JC$^*$-algebras}. There exist JB$^*$-algebras which are \emph{exceptional} in
the sense that they cannot be identified with a JB$^*$-subalgebra of a C$^*$-algebra, this is the case of the JB$^*$-algebra $H_3(\O)$ of all  $3\times 3$-hermitian matrices with entries in the algebra $\O$ of complex octonions (see, for example, \cite[\S 7.2]{hanche1984jordan}, \cite[\S 6.1 and 7.1]{Cabrera-Rodriguez-vol2} or \cite[\S 6.2 and 6.3]{Finite}).\smallskip

A JBW$^*$-algebra (respectively, a JW$^*$-algebra) is a JB$^*$-algebra (respectively, a JC$^*$-algebra) which is also a dual Banach space.\smallskip

JB$^*$-algebras are intrinsically connected with another mathematical object deeply studied in the literature. A \emph{JB-algebra} is a real Jordan algebra $J$ equipped with a complete norm satisfying \begin{equation}\label{eq axioms of JB-algebras} \|a^{2}\|=\|a\|^{2}, \hbox{ and } \|a^{2}\|\leq \|a^{2}+b^{2}\|\ \hbox{ for all } a,b\in J.
\end{equation}

In a celebrated lecture in St.\ Andrews in 1976 I. Kaplansky suggested the definition of JB$^*$-algebra and pointed out that the self-adjoint part $B_{sa}=\{x\in B : x^* =x\}$ of a JB$^*$-algebra is always a JB-algebra. One year later, J.D.M. Wright contributed one of the most influential results in the theory of JB$^*$-algebras by proving that the complexification of every JB-algebra is a JB$^*$-algebra (see \cite{Wright1977}). A \emph{JC-algebra} (respectively, a \emph{JW-algebra}) is a norm-closed (respectively, a weak$^*$-closed) real Jordan subalgebra of the hermitian part of a C$^*$-algebra (respectively, of a von Neumann algebra).\smallskip

Suppose $B$ is a unital JB$^*$-algebra. The smallest norm-closed real Jordan subalgebra $C(a)$ of $B_{sa}$ containing a self-adjoint element $a$ in $B$ and $1$ is associative. According to the usual notation, the \emph{spectrum of $a$} in $B$, denoted by $Sp(a)$, is the the set of all real $\lambda$ such that $a - \lambda 1$ does not have an inverse in $C(a)$ (cf. \cite[3.2.3]{hanche1984jordan}). If $B$ is not unital we consider the unitization of $B$ to define the spectrum of a self-adjoint element. It is known that the JB$^*$-subalgebra of $B$ generated by a single self-adjoint element $a\in B$ and the unit is isometrically JB$^*$-isomorphic to the commutative C$^*$-algebra $C(Sp(a))$, of all complex-valued continuous functions on $Sp(a)$ (see \cite[3.2.4. The spectral theorem]{hanche1984jordan}). An element $a\in B$ is called positive if $a=a^*$ and $Sp(a)\subseteq \mathbb{R}_{0}^+$ (cf. \cite[3.3.3]{hanche1984jordan}).\smallskip

Although there exist exceptional JB$^*$-algebras which cannot be embedded as JB$^*$-subalgebras of C$^*$-algebras, the JB$^*$-subalgebra of a JB$^*$-algebra $B$ generated by two hermitian elements (and the unit element) is a JC$^*$-algebra (compare Macdonald's and Shirshov-Cohn's theorems \cite[Theorems 2.4.13 and 2.4.14]{hanche1984jordan}, \cite[Corollary 2.2]{Wright1977} or \cite[Proposition 3.4.6]{Cabrera-Rodriguez-vol1}). Consequently, for each $x\in B$, the element $x\circ x^*$ is positive in $B$.\smallskip

We refer to the references \cite{hanche1984jordan,Cabrera-Rodriguez-vol1} and \cite{Cabrera-Rodriguez-vol2} for the basic background, notions and results on JB$^*$-algebras.\smallskip

C$^*$- and JB$^*$-algebras have been extensively employed as a framework for studying bounded symmetric domains in complex Banach spaces of infinite dimension, as an alternative notion to simply connected domains. The open unit ball of every C$^*$-algebra is a bounded symmetric domain (see \cite{harris1974bounded}) and the open unit balls of (unital) JB$^*$-algebras are, up to a biholomorphic mapping, those bounded symmetric domains
which have a realization as a tube domain, i.e. an upper half-plane (cf. \cite{BraKaUp78}). These examples do not exhaust all possible bounded symmetric domains in arbitrary complex Banach spaces, a strictly wider class of Banach spaces is actually required. The most conclusive result was obtained by W. Kaup who proved in 1983 that every bounded symmetric domain in a complex Banach space is biholomorphically equivalent to the open unit ball of a JB$^*$-triple \cite{kaup1983riemann}.\smallskip

A complex Banach space $E$ belongs to the class of {\em JB$^*$-triples} if it admits a triple product (i.e., a continuous mapping)  $\J{\cdot}{\cdot}{\cdot}:E^3\to E$ which is symmetric and bilinear in the outer variables and conjugate linear in the middle variable and satisfies the next algebraic and geometric axioms:
\begin{enumerate}[(JB$^*$-1)]
    \item $\J xy{\J abc}=\J{\J xya}bc-\J a{\J yxb}c+\J ab{\J xyc}$ for any $x,y,a,b,c\in E$ \hfill ({\em Jordan identity});
    \item For any $a\in E$ the operator $L(a,a):x\mapsto \J aax$ is hermitian with non-negative spectrum;
    \item $\norm{\J xxx}=\norm{x}^3$ for all $x\in E$ \hfill ({\em a Gelfand-Naimark type axiom}).
\end{enumerate}

C$^*$-algebras and JB$^*$-algebras belong to the wide list of examples of JB$^*$-triples when they are equipped with the triple products given by \begin{equation}\label{eq triple product JCstar and JBstar} \{a,b,c\} =\frac12 ( a b^* c + c b^* a),\hbox{ and } \{a,b,c\} = (a\circ b^*) \circ c +(c\circ b^*) \circ a - (a\circ c) \circ b^*,
\end{equation} respectively (see \cite[Theorem 3.3]{BraKaUp78} or \cite[Theorem 4.1.45]{Cabrera-Rodriguez-vol1}). The first triple product in \eqref{eq triple product JCstar and JBstar} induces a structure of JB$^*$-triple on every closed subspace of the space $B(H,K),$ of all bounded linear operators between complex Hilbert spaces $H$ and $K$, which is closed under this triple product. In particular, $B(H,K)$ and every complex Hilbert space are JB$^*$-triples with their canonical norms and the first triple product given in \eqref{eq triple product JCstar and JBstar}.\smallskip

In a JB$^*$-triple $E$ the triple product is contractive, that is, 
$$\|\{x,y,z\}\|\le\|x\| \|y\| \|z\| \quad\mbox{ for all } x,y,z \mbox{ in } E$$ (cf. \cite[Corollary 3]{Friedman-Russo-GN} or \cite[Corollary 7.1.7]{Cabrera-Rodriguez-vol2}, \cite[P.\ 215]{chubook}).\smallskip

A linear bijection between JB$^*$-triples is a triple isomorphism if and only if it is an isometry (cf. \cite[Proposition 5.5]{kaup1983riemann} or \cite[Theorems 3.1.7, 3.1.20]{chubook}). Thus, a complex Banach space admits a unique triple product under which it is a JB$^*$-triple. \smallskip

A JBW$^*$-triple is a JB$^*$-triple which is also a dual space. Every JBW$^*$-triple admits a unique (in the isometric sense) predual and its triple product is separately weak$^*$-continuous (see \cite{BaTi}, \cite[Theorems 5.7.20, 5.7.38]{Cabrera-Rodriguez-vol2}).\smallskip

Each idempotent $e$ in a Banach algebra $A$ produces a Peirce decomposition of $A$ as a sum of eigenspaces of the left and right multiplication operators by the idempotent $e$. A.A. Albert extended the classical Peirce decomposition to the setting of Jordan algebras in the middle of the last century. The notion of idempotent might mean nothing in a general JB$^*$-triple. The appropriate alternative is the concept of tripotent. An element $e$ in a JB$^*$-triple $E$ is a \emph{tripotent} if $\{e,e,e\}=e$. It is worth mentioning that when a C$^*$-algebra $A$ is regarded as a JB$^*$-triple with respect to the first triple product given in \eqref{eq triple product JCstar and JBstar}, an element $e\in A$ is a tripotent if and only if it is a partial isometry (i.e., $e e^*$, or equivalently $e^* e$, is a projection) in $A$.\smallskip

In case we fix a tripotent $e$ in a JB$^*$-triple $E$, the classical Peirce decomposition for associative and Jordan algebras extends to a \emph{Peirce decomposition} of $E$ associated with the eigenspaces of the mapping $L(e,e)$, whose eigenvalues are all contained in the set $\{0,\frac12,1\}$. For $j\in\{0,1,2\}$, the (linear) projection $P_{j} (e)$ of $E$ onto the eigenspace, $E_j(e)$, of $L(e, e)$ corresponding to the eigenvalue $\frac{j}{2},$ admits a concrete expression in terms of the triple product as follows: $$\begin{aligned} P_2 (e) &= L(e,e)(2 L(e,e) -id_{E})=Q(e)^2, \nonumber \\ P_1 (e) &= 4 L(e,e)(id_{E}-L(e,e)) =2\left(L(e,e)-Q(e)^2\right),  \hbox{ and } \nonumber \\ P_0 (e) &= (id_{E}-L(e,e)) (id_{E}-2 L(e,e)),
\end{aligned}$$ 
where $Q(e)(x)=\J exe$ for $x\in E$.
The projection $P_{j} (e)$ is known as the \emph{Peirce}-$j$ \emph{projection} associated with $e$. Peirce projections are all contractive (cf. \cite[Corollary 1.2]{Friedman-Russo}), and the JB$^*$-triple $E$ decomposes as the direct sum
$$E= E_{2} (e) \oplus E_{1} (e)\oplus E_0 (e),$$ which is termed the \emph{Peirce decomposition} of $E$ relative to $e$ (see \cite{Friedman-Russo}, \cite[Definition 1.2.37]{chubook} or \cite[Subsection 4.2.2]{Cabrera-Rodriguez-vol1} and \cite[Section 5.7]{Cabrera-Rodriguez-vol2} for more details). In the particular case in which $e$ is a tripotent (i.e. a partial isometry) in a C$^*$-algebra $A$ with initial projection $p_i= e^* e$ and final projection $p_f= e e^*$, the subspaces in the Peirce decomposition are precisely $$A_2(e) =p_f A p_i, \,  A_1(e) = p_f A (1-p_i)\oplus (1-p_f) A p_i,\,   A_0(e) = (1-p_f) A (1-p_i).     $$

A tripotent $e$ in a JB$^*$-triple $E$ is called \emph{complete} if $E_0 (e) =\{0\}$. We shall say that $e$ is \emph{unitary} if $E= E_2(e)$, or equivalently, if $\{e,e,x\}={x}$ for all $x\in E$. Obviously, every unitary is a complete tripotent, but the converse implication is not always true; consider for example a non-surjective isometry $e$ in $B(H)$. A non-zero tripotent $e$ satisfying $E_2(e) = \mathbb{C} e$ is called \emph{minimal}.\smallskip

Note that in a unital JB$^*$ algebra there is another definition of a unitary element  (cf.\ \cite[Definition 4.2.25]{Cabrera-Rodriguez-vol1}). However, it is equivalent
to the above-defined notion as witnessed by the following fact (where condition $(3)$ is the mentioned alternative definition). We will work solely with the notion of unitary tripotent defined above (i.e., with condition $(1)$ from the fact below) but we include these equivalences for the sake of completeness. 

\begin{fact}
Let $B$ be a unital JB$^*$-algebra and let $u\in B$. The following assertions are equivalent.
\begin{enumerate}[$(1)$]
    \item $u$ is a unitary tripotent, i.e., $u$ is a tripotent with $B_2(u)=B$.
    \item $u$ is a tripotent and $u\circ u^*=1$.
    \item $u\circ u^*=1$ and $u^2\circ u^*=u$, i.e., $u^*$ is the Jordan inverse of $u$.
\end{enumerate}
\end{fact}

\begin{proof}
The equivalence $(1)\Leftrightarrow(3)$ is proved in   \cite[Proposition 4.3]{BraKaUp78} (see also \cite[Theorem 4.2.28]{Cabrera-Rodriguez-vol1}). 

To prove the equivalence $(1)\Leftrightarrow(2)$ observe that assertion $(2)$ means that $1=\J uu1$, i.e., $1\in B_2(u)$. It remains to use \cite[Proposition 6.6]{hamhalter2019mwnc}.
\end{proof}

Complete tripotents in a JB$^*$-triple $E$ can be geometrically characterized since a norm-one element $e$ in $E$ is a complete tripotent if and only if it is an extreme point of its closed unit ball (cf. \cite[Lemma 4.1]{BraKaUp78}, \cite[Proposition 3.5]{kaup1977jordan} or \cite[Theorem 4.2.34]{Cabrera-Rodriguez-vol1}).
Consequently, every JBW$^*$-triple contains an abundant collection of complete tripotents.\smallskip

Given a unitary element $u$ in a JB$^*$-triple $E$, the latter becomes a unital JB$^*$-algebra with Jordan product and involution defined by
\begin{equation}\label{eq circ-star}
x\circ_{u} y=\J xuy\mbox{ and }x^{*_u}=\J uxu\qquad\mbox{for }x,y\in E,
\end{equation} see \cite[Theorem 4.1.55]{Cabrera-Rodriguez-vol1}. We even know that $u$ is the unit of this JB$^*$-algebra (i.e., $u\circ_u x=x$ for $x\in E$). Each tripotent $e$ in a JB$^*$-triple $E$ is a unitary in the JB$^*$-subtriple $E_2(e)$, and thus $(E_2(e),\circ_e,*_{e})$ is a unital JB$^*$-algebra. Therefore, since the triple product is uniquely determined by the structure of a JB$^*$-algebra, unital JB$^*$-algebras are in one-to-one correspondence with those JB$^*$-triples admitting a unitary element.\smallskip

A linear subspace $I$ of a JB$^*$-triple $E$ is called a \emph{triple ideal} or simply an \emph{ideal} of $E$ if $\J IEE\subset I$ and $\J EIE\subset I$ (see \cite{horn1987ideal}). Let $I, J$ be two ideals of $E$. We shall say that $I$ and $J$ are orthogonal if $I\cap J =\{0\}$ (and consequently $\{I,J,E\} = \{J,I,E\}=\{0\}$). It is known that every weak$^*$-closed ideal $I$ of a JBW$^*$-triple $M$ is orthogonally complemented, that is, there exists another weak$^*$-closed ideal $J$ of $M$ which is orthogonal to $I$ and $M = I \oplus^{\infty} J$ (see \cite[Theorem 4.2$(4)$ and Lemmata 4.3 and 4.4]{horn1987ideal}). For each weak$^*$-closed ideal $I$ of $M$, we shall denote by $P_{I}$ the natural projection of $M$ onto $I$.\label{eq ideals in JBW} Let us observe that, in this case $P_{I}$ is always a weak$^*$-continuous triple homomorphism.

\subsection{Positive functionals and prehilbertian seminorms}\label{subsec:seminorms}

As in the case of C$^*$-algebras, a functional $\phi$ in the dual space, $B^*$, of a JB$^*$-algebra $B$ is called positive if it maps positive elements to non-negative real numbers. We shall frequently apply that a functional $\phi$ in the dual space of a unital JB$^*$-algebra $B$ is positive if and only if $\|\phi\| = \phi (1)$ (cf. 
\cite[Lemma 1.2.2]{hanche1984jordan}). The same conclusion holds for functionals in the predual of a JBW$^*$-algebra.\smallskip

A positive normal functional $\varphi$ in the predual of a JBW$^*$-algebra $B$ is called \emph{faithful} if $\varphi (a) = 0$ for $a\geq 0$ in $B$ implies $a=0$.\smallskip

If $\phi$ is a positive functional in the dual of a C$^*$-algebra $A,$ and $1$ denotes the unit element in $A^{**}$, the mapping $$(a,b)\mapsto \phi \left(\frac{ a b^* + b^* a}{2} \right)= \phi \{a,b,1\} \ \ (a,b\in A)$$ is a positive semi-definite sesquilinear form on $A\times A$, whose associated prehilbertian seminorm is denoted by $\|x \|_{\phi} = (\phi \{x,x,1\})^{1/2}$. If we consider a positive functional $\phi$ in the dual of a JB$^*$-algebra $B$, the associated prehilbertian seminorm is defined by $\|x \|_{\phi}^2 = \phi \{x,x,1\} =\phi (x\circ x^*),$ where $1$ stands for the unit in $B^{**}$.\smallskip

The lacking of a local order or positive cone in a general JB$^*$-triple, and hence the lacking of positive functionals makes a bit more complicated the definition of appropriate prehilbertian seminorms. Namely, let $\varphi$ be a functional in the predual of JBW$^*$-triple $M$ and let $z$ be a norm-one element in $M$ satisfying $\varphi (z) = \|\varphi\|$. Proposition 1.2 in \cite{barton1987grothendieck} proves that the mapping $M\times M\to \mathbb{C}$, $(x,y)\mapsto \varphi\{x,y,z\}$ is a positive semi-definite sesquilinear form on $M$ which does not depend on the choice of the element $z$ (that is, $\varphi\{x,y,z\} = \varphi\{x,y,\tilde{z}\}$ for every $x,y\in M$ and every $\tilde{z}\in M$ with $\|\tilde{z}\|=1$ and $\varphi(\tilde{z})=\norm{\varphi}$, see \cite[Proposition 5.10.60]{Cabrera-Rodriguez-vol2}). The associated prehilbertian seminorm is denoted by $\norm{x}_{\varphi} =(\varphi\{x,x,z\})^{1/2}$ ($x\in M$). Since the triple product of every JB$^*$-triple is contractive it follows that \begin{equation}\label{eq seminorm inequality} \|x\|_\varphi\le\sqrt{\|\varphi\|}\|x\| \hbox{ for all } x\in M.\end{equation} If $\varphi$ is a non-zero functional in the dual $E^*$ of a JB$^*$-triple $E$, and we regard $E^*$ as the predual of the JBW$^*$-triple $E^{**}$, the prehilbertian seminorm $\norm{\cdot}_{\varphi}$ on $E^{**}$ acts on $E$ by mere restriction.\smallskip

\subsection{Comparison theory of projections and tripotents} Two projections $p,q$ in a C$^*$-algebra $A$ (respectively, in a JB$^*$-algebra $B$) are said to be orthogonal ($p\perp q$ in short) if $ p q = 0$ (respectively, $p\circ q =0$). The relation ``being orthogonal'' can be used to define a natural partial ordering on the set of projections in $A$ (respectively, in $B$) defined by $ p\leq q$ if $q-p$ is a projection and $q-p \perp p$. We write $p < q$ if $p\leq q$ and $p\neq q$.\smallskip

Two tripotents $e$ and $u$ in a JB$^*$-triple $E$ are called \emph{orthogonal} ($e\perp u$ in short) if $\{e,e,u\} = 0$ (equivalently, $u\in M_0 (e)$). It is known that $e\perp u$ if and only if any of the following equivalent reformulations holds:
\begin{enumerate}[$(1)$]
\item $e\in E_0(u)$.
\item $E_2(u)\subset E_0(e)$.
\item $L(u,e)=0$.
\item $L(e,u)=0$.
\item Both $u+e$ and $u-e$ are tripotents.
\item $\{u,u,e\} =0$.
\end{enumerate}
For proofs see \cite[Lemma 3.9]{loos1977bounded}, \cite[Proposition 6.7]{hamhalter2019mwnc} or \cite[Lemma 2.1]{Finite}. The induced partial order defined by this orthogonality on the set of tripotents is given by $e\leq u$ if $u-e$ is a tripotent with $u-e \perp e$.\smallskip

Let $\varphi$ be a non-zero functional in the predual of a JBW$^*$-triple $M$. By Proposition 2 in \cite{Friedman-Russo} (or \cite[Proposition 5.10.57]{Cabrera-Rodriguez-vol2}) there exists a unique tripotent $s(\varphi)\in M$, called the \emph{support tripotent} of $\varphi$, such that $\varphi=\varphi\circ P_2(s(\varphi))$, and $\varphi|_{M_2(s(\varphi))}$ is a faithful positive functional on the JBW$^*$-algebra $M_2(s(\varphi))$. In particular, $\norm{x}_{\varphi}^2=\varphi\{x,x,s(\varphi)\}$ for all $x\in M$.\smallskip

The support tripotent of a non-zero functional $\varphi$ in the predual of a JBW$^*$-triple $M$ is also  the smallest tripotent in $M$ at which $\varphi$ attains its norm, that is, \begin{equation}\label{eq minimality of the support tripotent}  \hbox{$\varphi (u) =\|\varphi\|$ for some tripotent $u\in M$} \Rightarrow s(\varphi)\leq u.
\end{equation}
Namely, the element $P_2(s(\varphi))(u)$ lies in the unit ball of $M_2(s(\varphi))$ because $P_2(s(\varphi))$ is contractive.  Since $\varphi = \varphi\circ P_2(s(\varphi))$ and $\varphi|_{M_2(s(\varphi))}$
is a faithful functional on the JBW$^*$-algebra $M_2(s(\varphi))$, we deduce that $P_2(s(\varphi)) (u) = s(\varphi)$. It follows from \cite[Lemma 1.6 or Corollary 1.7]{Friedman-Russo} that $s(\varphi)\leq u.$ Actually the previous arguments prove  \begin{equation}\label{eq minimality of the support tripotent for elements}  \hbox{$\varphi (a) =\|\varphi\|$ for some element $a\in M$ with $\|a\|=1$} \Rightarrow a= s(\varphi)+P_0 (s(\varphi)) (a).
\end{equation}

Two projections $p$ and $q$ in a von Neumann algebra $W$ are called \emph{(Murray-von Neumann) equivalent} (written $p\sim q$) if there is a partial isometry $e\in W$ whose initial projection is $p$ and whose final projection is $q$. This Murray-von Neumann equivalence is employed to classify projections and von Neumann algebras in terms of their properties. For example a projection $p$ in $W$ is said to be \emph{finite} if there is no projection $q< p$ that is equivalent to $p$. For example, all finite-dimensional projections in $B(H)$ are finite, but the identity operator on $H$ is not finite when $H$ is an infinite-dimensional complex Hilbert space. The von Neumann algebra $W$ is called \emph{finite} if its unit element is a finite projection. The set of all finite projections in the sense of Murray-von Neumann in $W$ forms a (modular) sublattice of the set of all projections in $W$ (see e.g. \cite[Theorem V.1.37]{Tak}). We recall that a projection $p$ in $W$ is {\em infinite} if it is not finite, and {\em properly infinite} if $p\ne 0$ and $zp$ is infinite whenever $z$ is a central projection such that $zp\ne0$ (cf. \cite[Definition V.1.15]{Tak}).\smallskip

In the setting of JBW$^*$-algebras the notion of finiteness was replaced by the concept of modularity, and the Murray-von Neumann equivalence by the relation ``being equivalent by symmetries'', that is, two projections $p,q$ in a JBW$^*$-algebra $N$ are called equivalent (by symmetries) (denoted by $p\stackrel{s}{\sim} q$) if there is a finite set $s_l, s_2 ,\ldots, s_n$ of self-adjoint symmetries (i.e. $s_j =1-2p_j$ for certain projections $p_j$) such that $Q(s_1)\cdots Q(s_n) (p) =q$, where $Q(s_j) (x) =\{s_j,x,s_j\} = 2 (s_j \circ x) \circ s_j - s_j^2 \circ x$ for all $x\in N$ (cf. \cite[\S 10]{Topping1965}, \cite[5.1.4]{hanche1984jordan}, \cite[\S 3]{AlfsenShultzGeometry2003} or \cite[\S 7.1]{Finite}). Unlike Murray-von Neumann equivalence, $p \stackrel{s}{\sim} q$ in $N$ implies $1-p \stackrel{s}{\sim} 1-q$. When $M$ is a von Neumann algebra regarded as a JBW$^*$-algebra, and $p,q$ are projections in $M$, $p\stackrel{s}{\sim}q$ if and only if $p$ and $q$ are unitarily equivalent, i.e. there exists a unitary $u\in M$ such that $u p u^* = q$ (see \cite[Proposition 6.56]{AlfsenShultzStateSpace2001}).
In particular,  $p\stackrel{s}{\sim} q$ implies $p\sim q$.\smallskip

In a recent contribution we study the notion of finiteness in JBW$^*$-algebras and JBW$^*$-triples from a geometric point of view. In the setting of von Neumann algebras, the results by H. Choda, Y. Kijima, and Y. Nakagami assert that a von Neumann algebra $W$ is finite if and only if all the extreme points of its closed unit ball are unitary (see  \cite[Theorem 2]{ChodaKijimaNak69} or \cite[Proof of Theorem 4]{mil}). Therefore, a projection $p$ in $W$ is finite if and only if every extreme point of the closed unit ball of $pWp$ is a unitary in the latter von Neumann algebra. This is the motivation for the notion of finiteness introduced in \cite{Finite}. According to the just quoted reference, a tripotent $e$ in a JBW$^*$-triple $M$ is called
\begin{enumerate}[$\bullet$]
\item {\em finite} if any tripotent $u\in M_2(e)$ which is complete in $M_2(e)$ is already unitary in $M_2(e)$;
\item {\em infinite} if it is not finite;
\item {\em properly infinite} if $e\ne 0$ and for each weak$^*$-closed ideal $I$ of $M$ the tripotent $P_I (e)$ is infinite whenever it is nonzero.
\end{enumerate} If any tripotent in $M$ is finite, we say that $M$ itself is {\em finite}. Finite-dimensional JBW$^*$-triples are always finite \cite[Proposition 3.4]{Finite}. The JBW$^*$-triple $M$ is said to be {\em infinite} if it is not finite. Finally, $M$ is {\em properly infinite} if each nonzero weak$^*$-closed ideal of $M$ is infinite.\smallskip

Every JBW$^*$-triple decomposes as an orthogonal sum of weak$^*$-closed ideals $M_1,$ $M_2$, $M_3$ and $M_4,$ where $M_1$ is a finite JBW$^*$-algebra, $M_2$ is either a trivial space or a properly infinite JBW$^*$-algebra, $M_3$ is a finite JBW$^*$-triple with no nonzero direct summand isomorphic to a JBW$^*$-algebra, and $M_4$ is either a trivial space or $M_4=qV_4$, where $V_4$ is a von Neumann algebra, $q\in V_4$ is a properly infinite projection such that $qV_4$ has no direct summand isomorphic to a JBW$^*$-algebra; we further know that $M_4$ is properly infinite in case that it is not zero (see \cite[Theorem 7.1]{Finite} where a more detailed description is presented). This decomposition applies in the particular case in which $M$ is a JBW$^*$-algebra with the appropriate modifications and simplifications on the summands to avoid those which are not JBW$^*$-algebras.\smallskip

In a von Neumann algebra $W$ the two notions of finiteness coincide for projections (see \cite[Lemma 3.2$(a)$]{Finite}). Every modular projection in a JBW$^*$-algebra is a finite tripotent in the sense above, but the reciprocal is not always true (cf. \cite[Lemma 7.12 and Remark 7.13]{Finite}).\smallskip

Finite JBW$^*$-triples enjoy formidable properties. For example, for each finite tripotent $u$ in a JBW$^*$-algebra $M$ there is a unitary element $e\in  M$ with $u \leq e$ (cf. \cite[Proposition 7.5]{Finite}). More details and properties can be found in \cite{Finite}.\smallskip

A projection $p$ in a von Neumann algebra $W$ is called \emph{abelian} if the subalgebra $pW p$ is abelian (see \cite[Definition V.1.15]{Tak}). The von Neumann algebra $W$ is said to be of \emph{type I} or \emph{discrete} if every nonzero (central) projection contains a nonzero abelian subprojection \cite[Definition V.1.17]{Tak}. In the previous definition the word central can be relaxed (see, for example, \cite[Corollary 4.20]{stra-zsi}).\smallskip

A tripotent $e$ in a JB$^*$-triple is said to be \emph{abelian} if the JB$^*$-algebra $E_2(u)$ is associative, or equivalently, $(E_2(u),\circ_u,*_u)$ is a unital abelian C$^*$-algebra. Obviously, any minimal tripotent is abelian. We further know that every abelian tripotent is finite \cite[Lemma 3.2$(e)$]{Finite}.\smallskip

According to \cite{horn1987classification,horn1988classification} and \cite{horn1987ideal}, a JBW$^*$-triple $M$ is said to be of \emph{type $I$} (respectively, \emph{continuous}) if it coincides with the weak$^*$ closure of the span of all its abelian tripotents (respectively, it contains no non-zero abelian tripotents). Every JBW$^*$-triple can be written as the orthogonal sum of two weak$^*$-closed ideals $M_1$ and $M_2$ such that $M_1$ is of type $I$ and $M_2$ is continuous (any of these summands might be trivial). G. Horn and E. Neher established in \cite{horn1987classification,horn1988classification} structure results describing type $I$ and continuous JBW$^*$-triples. Concretely, every JBW$^*$-triple of type $I$ may be represented in the form
\begin{equation}\label{eq:representation of type I JBW* triples}
 \bigoplus_{j\in J}^{\ell_\infty}A_j\overline{\otimes}C_j, 
\end{equation}
where the $A_j$'s are abelian von Neumann algebras and the $C_j$'s are Cartan factors (the concrete definitions will be presented below in Section \ref{sec:type I}, the reader can also consult \cite{loos1977bounded, kaup1981klassifikation, kaup1997real} for details). To reassure the reader we shall simply note that every Cartan factor $C$ is a JBW$^*$-triple. In the case in which $C$ is a JW$^*$-subtriple of some $B(H)$ and $A$ is an abelian von Neumann algebra, the symbol $A\overline{\otimes}C$ denotes the weak$^*$-closure of the algebraic tensor product $A{\otimes}C$ in the von Neumann tensor product $A\overline{\otimes} B(H)$ (see \cite[Section IV.1]{Tak} and \cite[\S 1]{horn1987classification}). In the remaining cases $C$ is finite-dimensional and $A\overline{\otimes} C$ will stand for the completed injective tensor product (see \cite[Chapter 3]{ryan}).

\section{Majorizing certain seminorms}\label{sec:majorizing}

The main result will be proved using its dual version. The starting point is the following dual version of Theorem~\ref{T:triples}$(2)$.

\begin{thm}[{\cite[Theorem 3]{peralta2001grothendieck}}]\label{T:triples-dual}
Let $M$ be a JBW$^*$-triple, $H$ a Hilbert space and $T:M\to H$ a weak$^*$-to-weak continuous linear operator. Given $\varepsilon>0$, there are norm-one functionals $\varphi_1,\varphi_2\in M_*$ such that
     $$\norm{Tx}\le(\sqrt{2}+\varepsilon)\norm{T}\left(\norm{x}^2_{\varphi_1}+\varepsilon\norm{x}^2_{\varphi_2}\right)^{\frac12}\quad\mbox{ for }x\in M.$$
\end{thm}

We continue by recalling two results from \cite{HKPP-BF}. The first one is essentially the main result and easily implies Theorem~\ref{T:triples}$(3)$. The second one was used to prove one of the particular cases and we will use it several times as well.

\begin{prop}[{\cite[Theorem 2.4]{HKPP-BF}}]\label{P:2-1BF} Let $M$ be a JBW$^*$-triple. Then given any two functionals $\varphi_1,\varphi_2$ in $M_*$, there exists a norm-one functional $\psi\in M_*$ such that
$$\norm{x}_{\varphi_1}^2+\norm{x}_{\varphi_2}^2\le 2(\norm{\varphi_1}+\norm{\varphi_2})\cdot \norm{x}_\psi^2$$ for all $x\in M.$ 
\end{prop}

\begin{lemma}[{\cite[Proposition 3.2]{HKPP-BF}}]\label{L:rotation} Let $M$ be a JBW$^*$-triple and let $\varphi\in M_*$.
Assume that $p\in M$ is a tripotent such that $s(\varphi)\in M_2(p)$.
Then there exists a  functional $\tilde{\varphi}\in M_*$ such that {$\norm{\tilde{\varphi}}=\norm{\varphi}$}, 
$s(\tilde{\varphi})\le p$ and $\norm{x}_\varphi\le\sqrt{2}\norm{x}_{\tilde{\varphi}}$ for all $x\in M$.
\end{lemma}

The key step to prove our main result is the following proposition which says that for JBW$^*$-algebras a stronger version of Proposition~\ref{P:2-1BF} is achievable.

\begin{prop}\label{P:majorize 1+2+epsilon}
Let $M$ be a JBW$^*$-algebra.
Then given any two functionals $\varphi_1,\varphi_2$ in $M_*$ and $\varepsilon>0$, there exists a norm-one functional $\psi\in M_*$ such that
$$\norm{x}_{\varphi_1}^2+\norm{x}_{\varphi_2}^2\le (\norm{\varphi_1}+2\norm{\varphi_2}+\varepsilon) \norm{x}_\psi^2 \mbox{ for }x\in M.$$
\end{prop}

Using this proposition we will easily deduce the main result in Section~\ref{sec:proofs} below. Proposition \ref{P:majorize 1+2+epsilon} will be proved using the following result.

\begin{prop}\label{P:key decomposition alternative} Let $M$ be a JBW$^*$-algebra, $\varphi\in M_*$ and $\varepsilon>0$. Then there are a functional $\tilde\varphi\in M_*$
and a unitary element $w\in M$ such that $$\norm{\tilde\varphi}\le\norm{\varphi}, \quad s(\tilde\varphi)\le w \quad\mbox{ and }
\norm{\cdot}^2_\varphi\le(1+\ep)\norm{\cdot}^2_{\tilde\varphi}.$$
\end{prop}

This proposition will be proved at the beginning of Section~\ref{sec:proofs} using the results from Sections~\ref{sec:type I} and~\ref{sec:JW*}. Let us now show that it implies Proposition~\ref{P:majorize 1+2+epsilon}.

\begin{proof}[Proof of Proposition~\ref{P:majorize 1+2+epsilon} from Proposition~\ref{P:key decomposition alternative}.]
Assume that $M$ is a JBW$^*$-algebra, $\varphi_1,\varphi_2\in M_*$ and $\varepsilon>0$.  
Let $\tilde\varphi_1\in M_*$ and $w\in M$ correspond to $\varphi_1$ and $\frac{\varepsilon}{\norm{\varphi_1}}$ by Proposition~\ref{P:key decomposition alternative}. 
Since $w$ is unitary, we have $M_2(w)=M$, hence we may apply Lemma~\ref{L:rotation} to get $\psi_2\in M_*$ such that
$$s(\psi_{2})\le w, \ \norm{\psi_{2}}\le\norm{\varphi_{2}},\ \norm{\cdot}_{\varphi_{2
}}\le\sqrt{2}\norm{\cdot}_{\psi_{2}}.
$$
Then
$$\begin{aligned}
 \norm{\cdot}_{\varphi_1}^2+\norm{\cdot}_{\varphi_2}^2&\le
 \left(1+\frac{\varepsilon}{\norm{\varphi_1}}\right)\norm{\cdot}_{\tilde\varphi_{1}}^2+\norm{\cdot}_{\varphi_2}^2\le
 \left(1+\frac{\varepsilon}{\norm{\varphi_1}}\right)\norm{\cdot}_{\tilde\varphi_{1}}^2+2\norm{\cdot}_{\psi_2}^2\\
 &=\norm{\cdot}_{\left(1+\frac{\varepsilon}{\norm{\varphi_1}}\right)\tilde\varphi_{1}+2\psi_2}^2
 =\left(\left(1+\frac{\varepsilon}{\norm{\varphi_1}}\right)\norm{\tilde\varphi_{1}}+2\norm{\psi_2}\right)\norm{\cdot}_\psi^2,
 \end{aligned}
$$
where $$\psi=\frac{(1+\frac{\varepsilon}{\norm{\varphi_1}})\tilde\varphi_{1}+2\psi_2}{(1+\frac{\varepsilon}{\norm{\varphi_1}})\norm{\tilde\varphi_{1}}+2\norm{\psi_2}}.$$
(Note that the first equality follows from the fact that the support tripotents of both functionals are below $w$.)
Since the functionals $\tilde\varphi_{1}$ and $\psi_2$  attain their norms at $w$, we deduce that $\norm{\psi}=1$. It remains to observe
that
$$\left(1+\frac{\varepsilon}{\norm{\varphi_1}}\right)\norm{\tilde\varphi_{1}}+2\norm{\psi_2}\le \norm{\varphi_1}+\varepsilon+2\norm{\varphi_2}.$$
\end{proof}

\section{Finite or type I JBW$^*$-algebras}\label{sec:type I}

The aim of this section is to prove a stronger version of Proposition~\ref{P:key decomposition alternative} for a large subclass of JBW$^*$-algebras (see Proposition \ref{P:type I approx}). We follow the notation from \cite{Finite} recalled in Section \ref{sec: notation}.

Since in a finite JBW$^*$-algebra any tripotent is majorized by a unitary one (cf. \cite[Lemma 3.2(d)]{Finite}), we get the following observation.

\begin{obs}\label{obs:finite JBW* algebras}
Let $M$ be a finite JBW$^*$-algebra. Then Proposition~\ref{P:key decomposition alternative} holds for $M$ in a very strong version -- one can take $\tilde\varphi=\varphi$ and $\varepsilon =0$.
\end{obs}

There is a larger class of JBW$^*$-algebras for which we get a stronger and  canonical version of Proposition~\ref{P:key decomposition alternative}. The concrete result appears in the content of the following proposition. The exact relationship with Proposition~\ref{P:key decomposition alternative} will be explained in
Remark \ref{Rem} (1) below.

We first recall that, in the setting of JBW$^*$-triples, two normal functionals $\varphi$ and $\psi$ in the predual of a JBW$^*$-triple $M$ are called (\emph{algebraically}) \emph{orthogonal} (written $\varphi\perp \psi$) if their support tripotents are orthogonal in $M$---that is, $s(\varphi)\perp s(\psi)$ (cf. \cite{FriRu87,EdRu01}). It is shown in \cite[Lemma\  2.3]{FriRu87} (see also \cite[Theorem 5.4]{EdRu01}) that $\varphi,\psi\in M_*$ are orthogonal if and only if they are ``geometrically'' \emph{$L$-orthogonal} in $M_*$ i.e., $\|\varphi \pm \psi\| = \|\varphi\| + \|\psi\|$.
In particular $\norm{\cdot}_{\varphi+\psi}^2=\norm{\cdot}_\varphi^2 + \norm{\cdot}_\psi^2$ if $\varphi$ and $\psi$ are orthogonal because in this case $\varphi$, $\psi$ and $\varphi+\psi$ attain their respective norms at $s(\varphi)+s(\psi)$.

\begin{prop}\label{P:type I approx}
Let $M$ be a JBW$^*$-algebra which is triple-isomorphic to a direct sum $M_1\oplus^{\ell_\infty}M_2$, where $M_1$ is a finite JBW$^*$-algebra and $M_2$ is a type I JBW$^*$-algebra. Let $\varphi\in M_*$ be arbitrary. Then for each $\varepsilon>0$ there are two functionals $\varphi_1,\varphi_2\in M_*$ such that
\begin{enumerate}[$(i)$]
    \item $\varphi=\varphi_1+\varphi_2$;
    \item $\varphi_1\perp\varphi_2$;
    \item $\norm{\varphi_2}<\varepsilon$;
    \item $s(\varphi_1)$ is a finite tripotent in $M$.
\end{enumerate}
\end{prop}

The rest of this section is devoted to prove Proposition~\ref{P:type I approx}. To this end we will use the following decomposition result which was essentially established in \cite{Finite}. Let us note that the concrete definition of a type 2 Cartan factor can be found in the next subsection.  

\begin{prop}\label{P:type I decomposition}
Let $M$ be a JBW$^*$-algebra which is triple-isomorphic to a direct sum $M_1\oplus^{\ell_\infty}M_2$, where $M_1$ is a finite JBW$^*$-algebra and $M_2$ is a type I JBW$^*$-algebra.
Then $M$ is triple-isomorphic to a JBW$^*$-algebra of the form
$$N\oplus^{\ell_\infty}\left(\bigoplus_{j\in J}L^\infty(\mu_j)\overline{\otimes}C_j\right)\oplus^{\ell_\infty}\left(\bigoplus_{\lambda\in \Lambda}L^\infty(\nu_\lambda)\overline{\otimes}B(H_\lambda)\right),$$
where
\begin{itemize}
    \item $N$ is a finite JBW$^*$-algebra;
    \item $J$ and $\Lambda$ are (possibly empty) sets;
    \item $\mu_j$'s and $\nu_\lambda$'s are probability measures;
    \item $C_j$ is an infinite-dimensional type 2 Cartan factor for each $j\in J$;
    \item $H_\lambda$ is an infinite-dimensional Hilbert space
    for each $\lambda\in\Lambda$.
\end{itemize}
\end{prop}

\begin{proof}
By \cite[Theorem 7.1]{Finite} $M$ is triple-isomorphic to $N\oplus^{\ell_\infty} N_1$, where $N$ is a finite JBW$^*$-algebra and $N_1$ is (either trivial or) a properly infinite JBW$^*$-algebra. By the same theorem $N_1$ is triple-isomorphic to 
$$\left(\bigoplus_{j\in J}L^\infty(\mu_j)\overline{\otimes}C_j\right)\oplus^{\ell_\infty}N_2,$$
where the first summand has the above-mentioned form and $N_2$ is (either trivial or) a properly infinite von Neumann algebra. Since by the assumptions $N_2$ is clearly of type I, we may conclude using \cite[Theorem V.1.27]{Tak}.
\end{proof}

We observe that the validity of Proposition~\ref{P:type I approx} is preserved by $\ell_\infty$-sums, so it is enough to prove it for the individual summands from Propostion~\ref{P:type I decomposition}. For the finite JBW$^*$-algebra $N$ we may use Observation~\ref{obs:finite JBW* algebras}.
We will prove the desired conclusion for the summands $L^\infty(\mu_j)\overline{\otimes}C_j$. For the remaining summands an easier version of the same proof works as we will explain below.

\subsection{The case of type 2 Cartan factors}\label{subsec:C2}

Let us start by recalling the definition of  type 2 Cartan factors. 
Let $H$ be a Hilbert space with a fixed orthonormal basis $(e_\gamma)_{\gamma\in \Gamma}$. Then $H$ is canonically represented as $\ell^2(\Gamma)$. For $\xi\in H$ let $\overline{\xi}$ be the coordinatewise complex conjugate of $\xi$. Further, for $x\in B(H)$ we denote by $x^t$ the operator defined by
$$x^t\xi=\overline{x^*\overline{\xi}},\qquad\xi\in H.$$
Then $x^t$ is the transpose of $x$ with respect to the fixed orthonormal basis, i.e., 
$$\ip{x^t e_\gamma}{e_\delta}=\ip{x e_\delta}{e_\gamma}\mbox{ for }\gamma,\delta\in\Gamma$$
(see, e.g., \cite[Section 5.3]{Finite} for the easy computation). Then
$$B(H)_s=\{x\in B(H)\setsep x^t=x\}\mbox{ and }B(H)_a=\{x\in B(H)\setsep x^t=-x\}$$
are the so-called \emph{Cartan factors} of \emph{type 3} and \emph{2}, respectively. They are formed by operators with symmetric (antisymmetric, respectively) `representing matrices' with respect to the fixed orthonormal basis. We will deal with the second case, i.e., with `antisymmetric operators'.

So, assume that $H$ has infinite dimension (or, equivalently, $\Gamma$ is an infinite set). Let $M=B(H)_a$.
Define $\pi:B(H)\to M$ by $\pi(x)=\frac12(x-x^t)$. Then $\pi$ is a norm-one projection which is moreover weak$^*$-to-weak$^*$ continuous.  Hence $\pi_*:M_*\to B(H)_*$ defined by $\pi_*(\varphi)=\varphi\circ\pi$ is an isometric injection. Moreover
$$\begin{aligned}\pi_*(M_*)&=\{\varphi\in B(H)_*\setsep \varphi(x^t)=-\varphi(x)\mbox{ for }x\in B(H)\}\\&=\{\varphi\in B(H)_*\setsep \varphi|_{B(H)_s}=0\}.\end{aligned}$$

Recall that $B(H)_*$ is isometric to the space of nuclear operators $N(H)$ via the trace duality (cf. \cite[Theorem II.1.8]{Tak}). Moreover, any $y\in N(H)$ is represented in the form
$$y=\sum_{k\geq 1} \lambda_k\ip{\cdot}{\eta_k}\xi_k$$
where $(\xi_k)$ and $(\eta_k)$ are orthonormal sequences in $H$ and the $\lambda_k$ are positive numbers with $\displaystyle\sum_{k\geq 1}\lambda_k=\norm{y}_N$.
Then clearly
$$y^*=\sum_{k\geq 1} \lambda_k\ip{\cdot}{\xi_k}\eta_k,$$
hence for any $\xi\in H$ we have
$$y^t\xi=\overline{y^*\overline{\xi}}
=\overline{\sum_{k\geq 1}  {\lambda_k} \ip{\overline{\xi}}{\xi_k}\eta_k}
=\sum_{k\geq 1}\lambda_k\ip{\xi}{\overline{\xi_k}}\overline{\eta_k},$$
thus
$$y^t=\sum_{k\geq 1} \lambda_k\ip{\cdot}{\overline{\xi_k}}\overline{\eta_k}.$$
In particular \begin{equation}\label{eq transpose preserves traces} \tr{y^t}=\sum_{k\geq 1} \lambda_k \ip{\overline{\eta_k}}{\overline{\xi_k}}=\sum_{k\geq 1} \lambda_k\ip{\xi_k}{\eta_k}=\tr{y}.
\end{equation}
Hence, given $\varphi\in B(H)_*$ represented by $y\in N(H)$, the functional $\varphi^t(x)=\varphi(x^t)$, $x\in B(H)$ is represented by $y^t$. Indeed,
$$\varphi^t(x)=\varphi(x^t)=\tr{x^ty}=\tr{y^tx}=\tr{xy^t}\mbox{ for }x\in B(H).$$
It follows that 
$$\pi_*M_*=\{\varphi\in B(H)_*\setsep \varphi\mbox{ is represented by an antisymmetric nuclear operator}\}.$$

\begin{proof}[Proof of Proposition~\ref{P:type I approx} for $M=B(H)_a$]
Fix $\varphi\in M_*$ of norm one and $\varepsilon>0$. Let $u=s(\varphi)\in M$.
Set $\tilde\varphi=\pi_*\varphi$. Fix $y\in N(H)$ representing $\tilde\varphi$. Then
$$y=\sum_{k\geq 1} \lambda_k\ip{\cdot}{\eta_k}\xi_k$$
where $(\xi_k)$ and $(\eta_k)$ are orthonormal sequences in $H$ and the $\lambda_k$ are strictly positive numbers with $\displaystyle \sum_{k\geq 1}\lambda_k=1$. Observe that
$$s(\tilde\varphi)=\sum_{k\geq 1}\ip{\cdot}{\xi_k}\eta_k.$$
Moreover, since $y$ is antisymmetric, we deduce that $s(\tilde\varphi)$ is also antisymmetric. Indeed, by the above
we have
$$y=-y^t=-\sum_{k\geq 1} \lambda_k\ip{\cdot}{\overline{\xi_k}}\overline{\eta_k}.$$
Hence
$$s(\tilde\varphi)=-\sum_{k\geq 1}\ip{\cdot}{\overline{\eta_k}}\overline{\xi_k}=-s(\tilde\varphi)^t.$$

For $\delta>0$ set
$$y_\delta=\sum_{\lambda_k\ge\delta} \lambda_k\ip{\cdot}{\eta_k}\xi_k.$$
Then $y_\delta$ is a finite rank operator and
$$y_\delta^t=\sum_{\lambda_k\ge\delta}\lambda_k\ip{\cdot}{\overline{\xi_k}}\overline{\eta_k}.$$
By uniqueness of the nuclear representation (the sequence $(\lambda_k)$ is unique and for any fixed $\lambda>0$ the linear spans of those $\eta_k$, resp. $\xi_k$, for which $\lambda_k=\lambda$
are uniquely determined) we deduce that $y_\delta$ is antisymmetric and hence its support tripotent
$$u_\delta=\sum_{\lambda_k\ge\delta}\ip{\cdot}{\xi_k}\eta_k$$
is antisymmetric as well.

Fix $\delta>0$ such that $\sum_{\lambda_k<\delta}\lambda_k<\varepsilon$.
Then $\norm{y-y_\delta}_N<\varepsilon$. 

Let $\tilde\varphi_1$ be the functional represented by $y_\delta$ and $\tilde\varphi_2=\tilde\varphi-\tilde\varphi_1$ (i.e., the functional represented by $y-y_\delta$). Since $y_\delta$ is antisymmetric, both $\tilde\varphi_1$ and $\tilde\varphi_2$ belong to $\pi_*M_*$. Moreover, $s(\tilde\varphi_1)=u_\delta$  and $s(\tilde\varphi_2)=u-u_\delta$.
Since $u_\delta\perp u-u_\delta$, we deduce that $\tilde{\varphi}_1\perp\tilde{\varphi}_2$. Further, $u_\delta$ is a finite tripotent, being a finite rank partial isometry.

Since we are in $\pi_*M_*$, we have functionals $\varphi_1,\varphi_2\in M_*$ such that $\tilde{\varphi}_j=\pi_*\varphi_j$.
It is now clear that they provide the sought decomposition of $\varphi$.
\end{proof}

We have settled the case of $B(H)_a$. Note that for $M=B(H)$ the same proof works -- we just do not use the mapping $\pi$ and are not obliged to check the antisymmetry. The proof was done using the Schmidt decomposition of nuclear operators. To prove the result for the tensor product we will use a measurable version of Schmidt decomposition established in the following subsection.

\subsection{Measurable version of Schmidt decomposition}

In this subsection we are going to prove the following result (note that $K(H)$ denotes the C$^*$-algebra of compact operators on $H$).

\begin{thm}\label{T:measurable Schmidt} Let $H$ be a 
Hilbert space. Then there are sequences $(\lambda_n)_{n=0}^\infty$ and $(\uu_n)_{n=0}^\infty$ of mappings such that the following properties are fulfilled for $n\in\en$ and $x\in K(H)$:
\begin{enumerate}[$(a)$]
    \item $\lambda_n:K(H)\to[0,\infty)$ is a lower-semicontinuous mapping;
    \item $\lambda_{n+1}(x)<\lambda_n(x)$ whenever $x\in K(H)$ and $\lambda_n(x)>0$;
    \item $\uu_n:K(H)\to K(H)$ is a Borel measurable mapping;
    \item $\uu_n(x)$ is a finite rank partial isometry on $H$;
    \item $\uu_n(x)=0$ whenever $\lambda_n(x)=0$;
    \item The partial isometries $\uu_k(x)$, $k\in\en\cup\{0\}$, are pairwise orthogonal;
    \item
    $x=\sum_{n=0}^\infty\lambda_n(x)\uu_n(x),$ where the series converges in the operator norm.
\end{enumerate}
\end{thm}

Let us point out that the Borel measurability in this theorem and in the lemmata used in the proof is considered with respect to the norm topology. However, if $X$ is a separable Banach space, it is well known and easy to see that any norm open set is weakly $F_\sigma$, hence the norm Borel sets coincide with the weak Borel sets (cf. \cite[pages 74 and 75]{Kuo75}). This applies in particular to $H$,  $K(H)$ and  $K(H)\times H$ where $H$ is a separable Hilbert space.

The proof will be done in several steps contained in the following lemmata.

\begin{lemma}\label{L:measurability of singular numbers}
Let $H$ be a Hilbert space (not necessarily separable). For $x\in K(H)$ let $(\alpha_n(x))$ be the sequence of its singular numbers.
Moreover, let $(\lambda_n(x))$ be the strictly decreasing version of $(\alpha_n(x))$ (recall that the sequence $(\alpha_n(x))$ itself is non-increasing), completed by zeros if necessary. I.e.,
$$\lambda_n(x)=\begin{cases}
\alpha_k(x) & \mbox{ if }\card \{\alpha_0(x),\alpha_1(x),\dots,\alpha_k(x)\}=n+1\\
 0 & \mbox{ if such $k$ does not exist.}
\end{cases}$$
Then the following assertions are valid for each $n\in\en\cup\{0\}$.
\begin{enumerate}[$(i)$]
    \item $\alpha_n$ is a $1$-Lipschitz function on $K(H)$;
    \item $\lambda_n$ is a lower semicontinuous function on $K(H)$, in particular it is Borel measurable and of the first Baire class.
\end{enumerate}
\end{lemma}

\begin{proof}
$(i)$ This is proved in \cite[Corollary VI.1.6]{Gohberg90} and 
easily follows from the following well-known formula for singular numbers
$$\alpha_n(x)=\dist \Big(x, \Big\{y\in K(H)\setsep \dim yH\le n\Big\}\Big), \quad x\in K(H), n\in\en\cup\{0\}$$ (cf. \cite[Theorem VI.1.5]{Gohberg90}).

$(ii)$ Clearly $\lambda_n\ge0$. Moreover, for each $c>0$ we have
$\lambda_n(x)>c$ if and only if
$$\exists\, c_0>c_1>\dots>c_n>c_{n+1}=c,\, \exists\, k_0,k_1,\dots k_n\in \en\, \hbox{ such that }$$ 
$$ \alpha_{k_j} (x)\in (c_{j+1},c_j)\ \ \forall j\in\{0,\dots,n\}.$$

Since the functions $\alpha_k$ are continuous by $(i)$, $\{x\setsep \lambda_n(x)>c\}$ is open. Now the lower semicontinuity easily follows.

Finally, any lower semicontinuous function on a metric space is clearly $F_\sigma$-measurable, hence Borel measurable and also of the first Baire class (cf. \cite[Corollary 3.8(a)]{LMZ}).
\end{proof}

\begin{lemma}\label{L:measurable projections}
Let $H$ be a 
Hilbert space. For any $x\in K(H)_+$ and $n\in\en\cup\{0\}$ let $p_n(x)$ be the projection onto the eigenspace with respect to the eigenvalue $\lambda_n(x)$ provided $\lambda_n(x)>0$
and $p_n(x)=0$ otherwise. Then the mapping $p_n$ is Borel measurable.
\end{lemma}

\begin{proof}
We start by proving that the mapping $p_0$ is Borel measurable. For $x\in K(H)_+\setminus\{0\}$ we set
$$\psi(x)=\frac{x-\lambda_0(x)\cdot I}{2(\lambda_0(x)-\lambda_1(x))}+I.$$
Then the mapping $\psi:K(H)_+\setminus\{0\}\to B(H)_{sa}$ is Borel measurable (by Lemma~\ref{L:measurability of singular numbers}$(ii)$, note that for $x\in K(H)_+\setminus\{0\}$ we have $\lambda_0(x)>\lambda_1(x)$).

Moreover, since
$$x=\sum_{n\ge0} \lambda_n(x)p_n(x),$$
by the Hilbert-Schmidt theorem, we deduce that
$$\psi(x)=p_0(x)+\sum_{n\ge1} \frac{\lambda_0(x)-2\lambda_1(x)+\lambda_n(x)}{2(\lambda_0(x)-\lambda_1(x))}p_n(x) +\frac{\lambda_0(x)-2\lambda_1(x)}{2(\lambda_0(x)-\lambda_1(x))}(I-\sum_{n\ge0}p_n(x)),$$
hence
the spectrum of $\psi(x)$ is
$$\sigma(\psi(x))=\left\{1,\tfrac{\lambda_0(x)-2\lambda_1(x)}{2(\lambda_0(x)-\lambda_1(x))}\right\}\cup\left\{\tfrac{\lambda_0(x)-2\lambda_1(x)+\lambda_n(x)}{2(\lambda_0(x)-\lambda_1(x))}\setsep n\ge 1\right\}\subset\{1\}\cup(-\infty,\tfrac12].$$
It follows that $p_0(x)=f(\psi(x))$ whenever $f$ is a continuous function on $\er$ with $f=0$ on $(-\infty,\frac12]$ and $f(1)=1$. 

Since the mapping $y\mapsto f(y)$ is continuous on $B(H)_{sa}$ by \cite[Proposition I.4.10]{Tak}, we deduce that $p_0$ is a Borel measurable mapping.

Further, for $n\in\en$ we have
$$p_n(x)=\begin{cases}0,& \hbox{if } \lambda_n(x)=0,\\
p_0\left(x-\sum_{k=0}^{n-1}\lambda_k(x)p_k(x)\right),& \hbox{if } \lambda_n(x)>0,\end{cases}$$
hence by the obvious induction we see that $p_n$ is Borel measurable as well.
\end{proof}

\begin{proof}[Proof of Theorem~\ref{T:measurable Schmidt}]
Fix any $x\in K(H)$. Let $x=u(x)\abs{x}$ be the polar decomposition.
By the Hilbert-Schmidt theorem we have
$$\abs{x}=\sum_n \lambda_n(x)  \, p_n(\abs{x})$$
(note that $\lambda_n(x)=\lambda_n(\abs{x})$). Hence 
$$x=\sum_n \lambda_n(x) u(x) p_n(\abs{x})=\sum_n \lambda_n(x) u_n(x),$$
where $u_n(x)=u(x)p_n(\abs{x})$ are mutually orthogonal partial isometries (of finite rank). The mappings $\lambda_n$ are lower semicontinuous by Lemma~\ref{L:measurability of singular numbers}. 

Further, the assignment $x\mapsto\abs{x}=\sqrt{x^*x}$ is continuous by the properties of the functional calculus.  Indeed, the mapping $x\mapsto x^*x$ is obviously continuous and the mapping $y\mapsto\sqrt{y}$ is continuous on the positive cone of $K(H)$ by \cite[Proposition I.4.10]{Tak}.

Hence, we can deduce from Lemma~\ref{L:measurable projections} that the assignments $x\mapsto p_n(\abs{x})$ are Borel measurable.
Since $u_n(x)=0$ whenever $\lambda_n(x)=0$ and $u_n(x)=\frac1{\lambda_n(x)}xp_n(\abs{x})$ if $\lambda_n(x)>0$, it easily follows that the mapping $u_n$ is Borel measurable.
\end{proof}

\begin{prop}\label{P:measurable nuclear rep}
Let $H$ be a separable Hilbert space. Consider the mappings $\lambda_n$ and $u_n$ provided by Theorem~\ref{T:measurable Schmidt}
restricted to $N(H)$. Then $\lambda_n$ and $u_n$ are Borel measurable also with respect to the nuclear norm. Moreover, the series from assertion $(g)$ converges absolutely in the nuclear norm and, moreover,
$$\norm{x}=\sum_{n=0}\lambda_n\norm{u_n(x)}$$
where the norm is the nuclear one.
\end{prop}

\begin{proof}
The Borel measurability of $\lambda_n$ and $u_n$ follows from the continuity of the canonical inclusion of $N(H)$ into $K(H)$ together with Theorem~\ref{T:measurable Schmidt}. The rest follows from the Schmidt representation of nuclear operators.
\end{proof}

\subsection{Proof of Proposition~\ref{P:type I approx}}

Let us adopt the notation from Subsection~\ref{subsec:C2}. 
Moreover, let $\mu$ be a probability measure and $A=L^\infty(\mu)$. 
Set $W=A\overline{\otimes}B(H)$. Then $W$ is a von Neumann algebra canonically represented in $B(L^2(\mu,H))$ (for a detailed description 
see e.g. \cite[Section 5.3]{Finite}). Moreover, on $L^2(\mu,H)$ we have a canonical conjugation (the pointwise one -- recall that $H=\ell^2(\Gamma)$ is equipped with the coordinatewise conjugation). Therefore we have a natural transpose of any $x\in W$ defined by
$$x^t(\f)=\overline{x^*(\overline{\f})}, \qquad \f\in L^2(\mu,H).$$
Then we have a canonical identification
$$M=A\overline{\otimes}B(H)_a=W_a=\{x\in W\setsep x^t=-x\}.$$
Similarly as in Subsection~\ref{subsec:C2} we denote by $\pi$ the canonical projection of $W$ onto $M$, i.e., $x\mapsto\frac12(x-x^t)$.

Recall that, by \cite[Theorem IV.7.17]{Tak},  $W_*=L^1(\mu,N(H))$ (the Lebesgue-Bochner space). Since $\pi$ is a weak$^*$-weak$^*$ continuous norm-one projection, we have an isometric embedding $\pi_*:M_*\to W_*$ defined by $\pi_*\omega=\omega\circ \pi$.   Moreover, clearly
$$\pi_*(M_*)=\{\omega\in W_*\setsep \omega^t=-\omega\}.$$

\begin{lemma}
Assume that $\g\in L^1(\mu,N(H))=W_*$. Then the following assertions hold.
\begin{enumerate}[$(i)$]
    \item $\g^*(\omega)=(\g(\omega))^*$ $\mu$-a.e.,
    \item $\g^t(\omega)=(\g(\omega))^t$ $\mu$-a.e.
\end{enumerate}
\end{lemma}

\begin{proof} Let us start by explaining the meaning. On the left-hand side we consider the involution and transpose applied to $\g$ as to a functional on $W$, while on the right-hand side these operations are applied to the nuclear operators $\g(\omega)$. 

Observe that it is enough to prove the equality for $\g=\chi_E y$ (where $E$ is a measurable set and $y\in N(H))$ as functions of this form are linearly dense in $L^1(\mu,N(H))$, i.e., we want to prove
$$(\chi_E y)^*=\chi_E y^*\mbox{ and }(\chi_E y)^t=\chi_E y^t.$$
It is clear that the elements on the right-hand side belong to $L^1(\mu,N(H))=W_*$, so the equality may be proved as equality of functionals. Since these functionals are linear and weak$^*$-continuous on $W$, it is enough to prove the equality on the generators $f\otimes x$, $f\in L^\infty(\mu)$, $x\in B(H)$.

So, fix such $f$ and $x$ and recall that
$$(f\otimes x)^*=\overline{f}\otimes x^*\mbox{ and }(f\otimes x)^t=f\otimes x^t.$$
Indeed, the first equality follows from the very definition of the von Neumann tensor product, the second one is proved in the computation before Lemma 5.10 in \cite{Finite}.
Hence we have
$$\begin{aligned}
\ip{(\chi_E y)^*}{f\otimes x}&=\overline{\ip{\chi_Ey}{\overline{f}\otimes x^*}}=
\overline{\int_E\overline{f}\di\mu \cdot\tr{yx^*}}
=\int_E f\di\mu \cdot\tr{(yx^*)^*}\\&=\int_E f\di\mu \cdot\tr{xy^*}=\int_E f\di\mu \cdot\tr{y^*x}=\ip{\chi_E y^*}{f\otimes x}\end{aligned}$$
and, similarly, by \eqref{eq transpose preserves traces}, we get
$$\begin{aligned}
\ip{(\chi_E y)^t}{f\otimes x}&=\ip{\chi_Ey}{f\otimes x^t}=
\int_E f\di\mu \cdot\tr{yx^t}
=\int_E f\di\mu \cdot\tr{(yx^t)^t}\\&=\int_E f\di\mu \cdot\tr{xy^t}=\int_E f\di\mu \cdot\tr{y^tx}=\ip{\chi_E y^t}{f\otimes x}.\end{aligned}$$
\end{proof}

It easily follows that
$$\pi_*(M_*)=L^1(\mu,N(H)_a).$$

\begin{lemma}\label{L:sepred} Let $\g\in L^1(\mu,N(H))=W_*$. Then the following assertions hold.
\begin{enumerate}[$(i)$]
    \item $\ip{f\otimes x}{\g}=\int f(\omega)\tr{x\g(\omega)}\di\mu(\omega)$ for $f\in L^\infty(\mu)$ and $x\in B(H)$.
    \item There is a projection $p\in B(H)$ with separable range such that $p\g(\omega)p=\g(\omega)$ $\mu$-a.e. In this case we have    $(1\otimes p)\g(1\otimes p)=\g$, i.e.,
    $$\ip{T}{\g}=\ip{(1\otimes p)T(1\otimes p)}{\g}\mbox{ for }T\in W.$$
\end{enumerate}
\end{lemma}

\begin{proof}
$(i)$ Fix $f\in L^\infty(\mu)$ and $x\in B(H)$. Consider both the left hand side and the right hand side as functionals depending on $\g$. Since both functionals are linear and continuous on $L^1(\mu,N(H))$, it is enough to prove the equality for $\g=\chi_E y$ where $E$ is a measurable set and $y\in N(H)$. In this case we have
$$\ip{f\otimes x}{\chi_E y}=\int_E f\di\mu \tr{xy},$$
so the equality holds.

$(ii)$ Note that $\g$ is essentially separably-valued, so there is a separable subspace $Y\subset N(H)$ such that $\g(\omega)\in Y$ $\mu$-a.e. Since for any $y\in N(H)$ there is a projection $q$ with separable range with $qyq=y$ (due the the Schmidt representation),
the existence of $p$ easily follows.

To prove the last equality it is enough to verify it for the generators $T=f\otimes x$ and this easily follows from $(i)$.
\end{proof}

\begin{prop}\label{P:L1 measurable repr}
Let $\g\in L^1(\mu,N(H))$. Then there are a separable subspace $H_0\subset H$, a sequence $(\zeta_n)$ of nonnegative measurable functions and a sequence $(\uu_n)$ of measurable mappings with values in $K(H_0)$ such that the following holds for each $\omega$:
\begin{enumerate}[$(a)$]
       \item $\zeta_{n+1}(\omega)<\zeta_n(\omega)$ whenever $\zeta_n(\omega)>0$;
        \item $\uu_n(\omega)$ is a finite rank partial isometry on $H_0$;
    \item $\uu_n(\omega)=0$ whenever $\zeta_n(x)=0$;
    \item the partial isometries $\uu_k(\omega)$, $k\in\en\cup\{0\}$, are pairwise orthogonal;
    \item $\g=\sum_{n=0}^\infty\zeta_n\uu_n$ where the series converges absolutely almost everywhere and also in the  norm of $L^1(\mu,N(H))$.
\end{enumerate}
\end{prop}

\begin{proof}
Let $p\in B(H)$ be a projection with separable range provided by
Lemma \ref{L:sepred}$(ii)$ and set $H_0=pH$. Let $(\lambda_n)$ and $(u_n)$ be the mappings provided by Theorem~\ref{T:measurable Schmidt}. 
Let $\uu_n(\omega)=u_n(\g(\omega))$ and $\zeta_n(\omega)=\lambda_n(\g(\omega))$. Then these functions are measurable due to measurability of $\g$ and Proposition~\ref{P:measurable nuclear rep}.
Assertions $(a)-(d)$ now follow from Theorem~\ref{T:measurable Schmidt}.

By Proposition~\ref{P:measurable nuclear rep} we get the first statement of $(e)$ and, moreover,
$$\sum_n\norm{\zeta_n(\omega)\uu_n(\omega)}=\norm{\g(\omega)}\ \mu\mbox{-a.e.},$$
hence the convergence holds also in the norm of $L^1(\mu,N(H))$, by the Lebesgue dominated convergence theorem for Bochner integral.
\end{proof}

Set
$$W_0=\{\f:\Omega\to B(H)\setsep \f\mbox{ is bounded, measurable and has separable range}
\}.$$
 By a measurable function we mean a strongly measurable one, i.e., an almost everywhere limit of simple functions. However, note that weak measurability is equivalent in this case by Pettis measurability theorem
 as we consider only functions with separable range.
 
Then $W_0$ is clearly a C$^*$-algebra when equipped with the pointwise operation and supremum norm. 

We remark that the following lemma seems to be close to the results of \cite[Section IV.7]{Tak}. However, it is not clear how to apply these results in our situation, so we give the proofs.

\begin{lemma}
For  $\f\in W_0$ and $\h\in L^2(\mu,H)$ define the function $T_{\f}\h$ by the formula
 $$T_{\f}\h (\omega)= \f(\omega)(\h(\omega)),\quad \omega\in\Omega.$$
\begin{enumerate}[$(i)$]
    \item For each $\f\in W_0$ the mapping $T_{\f}$ is a bounded linear operator on $L^2(\mu,H)$ which belongs to $W$ and satisfies  $\norm{T_{\f}}\le\norm{\f}_\infty$.
    \item If $\f\in W_0$ and $\g\in W_*=L^1(\mu,N(H))$, then
    $$\ip{T_{\f}}{\g}=\int\tr{\f(\omega)\g(\omega)}\di\mu(\omega).$$
    \item $T_{\f}$ is a partial isometry (a projection) in $W$ whenever $\f(\omega)$ is a partial isometry (a projection) $\mu$-a.e.
    \item If $\g\in L^1(\mu,N(H))$ is represented as in Proposition~\ref{P:L1 measurable repr}$(e)$, then $s(\g)\le \sum_n T_{\uu_n^*}$ where series converges  in the SOT topology in $W$.
\end{enumerate}
\end{lemma}

\begin{proof}
$(i)$ It is clear that the mapping $\h\mapsto T_{\f}\h$ is a linear mapping assigning to each $H$-valued function another $H$-valued function. Moreover,
$$\norm{T_{\f}\h(\omega)}=\norm{\f(\omega)(\h(\omega))}\le\norm{\f(\omega)}\norm{\h(\omega)}\le\norm{\f}_\infty\norm{\h(\omega)}.$$
In particular, if a sequence $(\h_n)$ converges almost everywhere to a function $\h$, then $(T_{\f}\h_n)$ converges almost everywhere to $T_{\f}\h$. It follows that $T_{\f}$ is well defined on $L^2(\mu,H)$
 (in the sense that if $\h_1=\h_2$ a.e., then $T_{\f}\h_1=T_{\f}\h_2$ a.e.).

The next step is to observe that $T_{\f}\h$ is measurable whenever $\h$ is measurable. This is easy for simple functions. Further,
any measurable function is an a.e. limit of a sequence of simple functions, hence the measurability follows by the above.

Further, it follows from the above inequality that $\norm{T_{\f}\h}_2\le \norm{\f}_\infty\norm{\h}_2$, thus $\norm{T_{\f}}\le\norm{\f}_\infty$. Finally,  by \cite[Lemma 5.12]{Finite} we get that $T_{\f}\in W$.

$(ii)$ Let us first show that $\f\g\in L^1(\mu,N(H))$ whenever $\f\in W_0$ and $\g\in L^1(\mu,N(H))$. By the obvious inequalities the only thing to be proved is measurability of this mapping. 
This is easy if $\g$ is a simple function. The general case follows from the facts that any measurable function is an a.e. limit of simple functions and that measurability is preserved by a.e. limits of sequences.

It remains to prove the equality.
Since the functions from $W_0$ are separably valued, countably valued functions are dense in $W_0$. So, it is enough to prove the equality for countably valued functions. To this end let
$$\f=\sum_{k\in\en}\chi_{E_k}x_k,$$
where $(E_k)$ is a disjoint sequence of measurable sets and $(x_k)$ is a bounded sequence in $B(H)$. For any $\h\in L^2(\mu,H)$ we have
$$T_{\f}\h(\omega)=\sum_{k\in\en}\chi_{E_k}(\omega)x_k(\h(\omega)),\qquad \omega\in \Omega.$$
Since $T_{\f}\h\in L^2(\mu,H)$ by $(i)$ and the sets $E_k$ are pairwise disjoint, we deduce that
$$T_{\f}\h=\sum_{k\in\en}T_{\chi_{E_k}x_k}\h,$$
where the series converges in $L^2(\mu,H)$.
Since this holds for any $\h\in L^2(\mu,H)$, we deduce that
$$T_{\f}=\sum_{k\in\en}T_{\chi_{E_k}x_k}$$
unconditionally in the SOT topology, hence also in the weak$^*$ topology of $W$. Thus, for any $\g\in W_*=L^1(\mu,N(H))$ we get
$$\ip{T_{\f}}{\g}=\sum_{k\in\en}\ip{T_{\chi_{E_k}x_k}}{\g}=\sum_{k\in\en}\int_{E_k}\tr{x_k\g(\omega)}\di\mu(\omega)=\int\tr{\f(\omega)\g(\omega)}\di\mu(\omega),
$$
where in the second equality we used Lemma~\ref{L:sepred}$(i)$.

$(iii)$ This is obvious as the mapping $\f\mapsto T_{\f}$ is clearly a $*$-homomorphism of $W_0$ into $W$.

$(iv)$ First observe that the mappings $\uu_n^*$ belong to $W_0$. Indeed, by Proposition~\ref{P:L1 measurable repr} the mapping $\uu_n$ is measurable and has separable range (as $K(H_0)$ is separable). Moreover, $\norm{\uu_n}_\infty\le1$ for each $n\in\en$. These properties are shared by $\uu_n^*$, hence $\uu_n^*\in W_0$.

By $(iii)$ we deduce that $T_{\uu_n^*}$ is a partial isometry for any $n\in\en$. Moreover, these partial isometries are pairwise orthogonal (cf. property $(d)$ from Proposition~\ref{P:L1 measurable repr}), hence $U=\sum_n T_{\uu_n^*}$ is a well-defined partial isometry in $W$. Moreover, by taking $\g$ as in Proposition~\ref{P:L1 measurable repr}$(e)$, we have 
$$\begin{aligned}
\ip{U}{\g}&=\sum_{n=0}^\infty\ip{T_{\uu_n^*}}{\g}=\sum_{n=0}^\infty\int \tr{\uu_n^*(\omega)\g(\omega)}\di\mu\omega\\&=
\sum_{n=0}^\infty\int \zeta_n(\omega)\tr{\uu_n^*(\omega)\uu_n(\omega)}\di\mu(\omega)
\\&=\int\sum_{n=0}^\infty \zeta_n(\omega)\tr{\uu_n^*(\omega)\uu_n(\omega)}\di\mu(\omega)=\int \norm{\g(\omega)}\di\mu(\omega)=\norm{\g},\end{aligned}$$
thus $s(\g)\le U$. 
\end{proof}

\begin{proof}[Proof of Proposition~\ref{P:type I approx} for $A\overline{\otimes}B(H)_a$]
Fix any $\g\in M_*=L^1(\mu,N(H)_a)$ and $\varepsilon>0$. Fix its representation from Proposition~\ref{P:L1 measurable repr}. Fix $N\in\en$ such that
$$\norm{\sum_{n>N}\zeta_n\uu_n}<\varepsilon.$$
This is possible by the convergence established in Proposition~\ref{P:L1 measurable repr}.
Note that
$$-\g=\g^t=\sum_{n=1}^\infty\zeta_n\uu_n^t,$$
hence 
$\uu_n^t=-\uu_n^t$.  (Note that the representation from Proposition~\ref{P:L1 measurable repr} is unique due to the uniqueness of the Hilbert-Schmidt representation). 
Let 
$$\g_1=\sum_{n=1}^N\zeta_n\uu_n.$$
Then $\g_1\in M_*$ as $\g_1^t=-\g_1$.
Further, let $$\vv=\sum_{n=1}^N \uu_n.$$ 
We have $\g-\g_1\perp \g_1$ as 
$$s(\g_1)\le T_{\vv^*} \mbox{ and }s(\g-\g_1)\le \sum_{n>N}T_{\uu_n^*}$$
and the two tripotents on the right-hand sides are orthogonal. Moreover,
$T_{\vv^*}$ is a finite tripotent in $M$ by \cite[Proposition 5.31(i) and Lemma 5.16(ii)]{Finite}. 
This completes the proof.
\end{proof}

\begin{proof}[Proof of Proposition~\ref{P:type I approx} for $A\overline{\otimes}B(H)$] The proof is an easier version of the previous case. Fix $\g\in W_*=L^1(\mu,N(H))$ and $\varepsilon>0$. In the same way we find $N$ and define $\g_1$ and $\vv$. We omit the considerations of the transpose and antisymmetry. Finally, $T_{\vv^*}$ is a finite tripotent in $W$
by \cite[Proposition 4.7 and Lemma 5.16(ii)]{Finite}.
\end{proof}

\section{JW$^*$-algebras}\label{sec:JW*}

The aim of this section is to prove
the following proposition which will be used to prove Proposition~\ref{P:key decomposition alternative}.

\begin{prop}\label{P:key decomposition} Let $M$ be a JBW$^*$-algebra, $\varphi\in M_*$ and $\varepsilon>0$. Then there are functionals $\varphi_1,\varphi_2\in M_*$
and a unitary element $w\in M$ satisfying the following conditions.
\begin{enumerate}[$(i)$]
    \item\label{it:key decomposition i} $\norm{\varphi_1}\le\norm{\varphi}$;
    \item $\norm{\varphi_2}<\varepsilon$;
    \item $s(\varphi_1)\le w$;
    \item\label{it:key decomposition iv} $\norm{\cdot}_\varphi^2\le\norm{\cdot}_{\varphi_1}^2+\norm{\cdot}_{\varphi_2}^2$.
\end{enumerate}
\end{prop}

The proof will be done at the end of the section with the help of several lemmata.

We focus mainly on JW$^*$-algebras, i.e., on weak$^*$-closed Jordan $^*$-subalgebras of von Neumann algebras. To this end we recall some notation (cf. \cite[Section III.2]{Tak}).

 Let $A$ be a C$^*$-algebra and let $\phi\in A^*$.  Then we define
 functionals $a\phi$ and $\phi a$  by 
\begin{equation}
a\phi(x)=\phi(xa)
\quad\mbox{ and }\quad
 \phi a(x)=\phi(ax)  \quad\mbox{for } x\in A.
\end{equation}
Note that $a\phi,\phi a\in A^*$ and $\norm{a\phi}\le\norm{a}\norm{\phi}$, $\norm{\phi a}\le\norm{a}\norm{\phi}$.
We recall the natural isometric involution $\phi\mapsto\phi^*$ defined by $\phi^*(x)=\overline{\phi(x^*)}$.
Then clearly
$(a\phi)^*=\phi^*a^*$, $(\phi a)^*=a^*\phi^*$.

If $W$ is a von Neumann algebra and if $\phi\in W_*$, $a\in W$ then $a\phi, \phi a\in W_*$.
Further, given $\phi\in W_*$ we set $\abs{\phi}=s(\phi)\phi$ where $s(\phi)\in W$ is the support tripotent of $\phi$. Then $\phi=s(\varphi)^*\abs{\phi}$ is the polar decomposition of $\phi$ (cf. \cite[Section III.4]{Tak}).
More generally, if $a\in W$ is a norm-one element on which $\phi$ attains its norm then we have $\betr{\phi}=a\phi$, $\phi=a^*\betr{\phi}$, $\betr{\phi^*}=\phi a$ (cf. \eqref{eq minimality of the support tripotent for elements}).
Note that $\betr{\phi}=\betr{\phi}^*$ since $\betr{\phi}$ is positive.
All this is stable by small perturbations as witnessed by the following lemma.

\begin{lemma}[{\cite[Lemma 3.3]{pfitzner-jot}}]\label{l perturbation of functionals}
Let $A$ be a C$^*$-algebra, $\phi$ a functional on $A$ and $a,b$ in the unit ball of $A$.
Then
\begin{eqnarray}
\norm{\phi-a^*\betr{\phi}\;}
   &\leq&
({2\norm{\phi}})^{1/2}\,\,\betr{\,\norm{\phi} - \phi(a)}^{1/2}   \label{glA3_1}\\
\norm{\betr{\phi}-a\phi}
   &\leq&
({2\norm{\phi}})^{1/2}\,\,\betr{\,\norm{\phi} - \phi(a)}^{1/2}   \label{glA3_2}\\
\norm{\betr{\phi^*}-\phi a}
   &\leq&
({2\norm{\phi}})^{1/2}\,\,\betr{\,\norm{\phi} - \phi(a)}^{1/2}.   \label{glA3_3}
\end{eqnarray}
\end{lemma}\noindent
(As to \eqref{glA3_3}, which is not stated explicitly in \cite[Lemma 3.3]{pfitzner-jot}, note that it
follows easily from \eqref{glA3_2} by $\norm{\betr{\phi^*}-\phi a}=\norm{\betr{\phi^*}-a^*\phi^*}\le(2\norm{\phi^*})^{1/2}\,\,\betr{\norm{\phi^*} - \phi^*(a^*)}^{1/2}=({2\norm{\phi}})^{1/2}\,\,\betr{\,\norm{\phi} - \phi(a)}^{1/2}$.)

There is another way to obtain positive functionals: We can write $\phi=\phi_1-\phi_2+i(\phi_3-\phi_4)$ with positive  $\phi_k\in W_*$ ($k=1, 2, 3, 4$) such that
$\norm{\phi_k-\phi_{k+1}}=\norm{\phi_k}+\norm{\phi_{k+1}}\le\norm{\phi}$, $k=1, 3$ (cf. \cite[Theorem III.4.2]{Tak}).
Then we set 
$$[\phi]=\frac12\sum_{k=1}^4\phi_k=\frac12(\abs{\phi_1-\phi_2}+\abs{\phi_3-\phi_4})$$ 
and obtain that $[\phi]\in W_*$ is positive, $\norm{[\phi]}\le\norm{\phi}$
and $\betr{\phi(a)}\le2[\phi](a)$ for all positive $a\in W$.

Finally, let us remark that if $A$ is a C$^*$-algebra, then $A^{**}$ is a von Neumann algebra and $A^*=(A^{**})_*$, thus $\abs{\phi}$ and $[\phi]$ make sense also for continuous functionals on a C$^*$-algebra.

\begin{lemma}\label{l norm-one functional close enough to states at a unitary Cstar} Let $W$ be von Neumann algebra, let $w\in W$ be a unitary element and $\delta\in(0,1)$. Let $\phi\in W_*$ be a norm-one functional such that $\phi (w) >1-\delta$ (in particular, $\phi(w)\in\er$). Then $\psi := w^* |\phi|$ is a norm-one element of $W_*$ satisfying $\psi(w) =1$ 
and $\|\phi -\psi\| < \sqrt{2\delta}$.
\end{lemma}

\begin{proof}
On the one hand we have that $\|\psi \| \le  \| |\phi|\| = \|\phi\|=1$. On the other hand, since $\psi (w) = (w^* |\phi|) (w) = |\phi| (w  w^*) = |\phi| (1) = \| |\phi|\| = 1$ we deduce that $\norm{\psi}=1$.
Applying \eqref{glA3_1} of Lemma \ref{l perturbation of functionals} we obtain $$ \|\phi - w^* |\phi| \| \leq \sqrt{2 } \betr{ 1 - \phi(w)}^{1/2} \leq \sqrt{2 \delta},$$ which finishes the proof.
\end{proof}

We continue by extending the previous lemma to JW$^*$-algebras.

\begin{lemma}\label{l norm-one functional close enough to states at a unitary} Let $M$ be a JW$^*$-algebra, $w\in M$ a unitary element and $\delta\in(0,1)$. Let $\phi\in M_*$ be a norm-one functional such that $\phi (w) >1-\delta$ (in particular, $\phi(w)\in\er$). Then there exists a norm-one functional $\psi\in M_*$ satisfying $\psi(w) =1$
and $\|\phi -\psi\| < \sqrt{2\delta}$. 
\end{lemma}

\begin{proof} Let us assume that $M$ is a JW$^*$-subalgebra of a von Neumann algebra $W$.
Let $1$ denote the unit of $M$. Then $1$ is a projection in $W$, thus, up to replacing $W$ by $1W1$, we may assume that $M$ contains the unit of $W$.
\smallskip

We observe that $w$, being a unitary element in $M$, is unitary in $W$.  Let $\tilde{\phi}\in W_*$ be a norm-preserving extension of $\phi$ provided by \cite[Theorem]{Bun01}.
By hypothesis, 
$1-\delta < {\phi} (w)= \tilde{\phi} (w) \leq \|{\phi}\| = \|\tilde{\phi}\|=1$. Now, applying Lemma \ref{l norm-one functional close enough to states at a unitary Cstar} to $W$, $\tilde{\phi}\in W_*$ and the unitary $w$, we find a norm-one functional $\tilde{\psi}\in W_*$ satisfying $\tilde{\psi} (w) =1$ and $\|\tilde{\phi} -\tilde{\psi}\| < \sqrt{2\delta}$. Since $w\in M$ and $1= \tilde{\psi} (w)$, the functional $\psi = \tilde{\psi}|_{M}$ has norm-one, $\psi (w) =1$ and clearly $\|{\phi} -{\psi}\| < \sqrt{2\delta}$.
\end{proof}

\begin{lemma}\label{l ad-hoc 1 Jordan} Let $M$ be a JW$^*$-algebra, let $\phi\in M_*$ and $\delta>0$.
Suppose $a_1, a_2$ are two norm-one elements in $M$ such that $$ \betr{\norm{\phi}-\phi(a_k)}<\delta\norm{\phi}  \hbox{ for  } k=1,2.$$
Then there is a positive functional $\omega\in M_*$ satisfying $\norm{\omega}\le2\sqrt{2\delta}\norm{\phi}$ and
$$\betr{\phi\J xx{a_1-a_2}}\le4\norm{x}_\omega^2 \hbox{ for all } x\in M.$$
\end{lemma}

\begin{proof} Similarly as in the  proof of Lemma \ref{l norm-one functional close enough to states at a unitary}  we may assume that $M$ is a JW$^*$-subalgebra of a  von Neumann algebra $W$ containing the unit of $W$. \smallskip

Let $\tilde{\phi}\in W_*$ be a norm-preserving normal extension of $\phi$ (see \cite[Theorem]{Bun01}). Working in $W_*$ we set $\tilde\psi_l=a_1\tilde\phi-a_2\tilde\phi$ and $\tilde\psi_r=\tilde\phi a_1-\tilde\phi a_2$.
By \eqref{glA3_2} of Lemma \ref{l perturbation of functionals}
we have $\norm{\betr{\tilde\phi}-a_k\tilde\phi}\le\sqrt{2\delta}\norm{\tilde\phi}$ ($k=1,2$) hence $\norm{\tilde\psi_l}\le2\sqrt{2\delta}\norm{\tilde\phi}$.
Likewise we get $\norm{\tilde\psi_r}\le2\sqrt{2\delta}\norm{\tilde\phi}$ with \eqref{glA3_3} of Lemma \ref{l perturbation of functionals}.
Set $\tilde\omega=([\tilde\psi_l]+[\tilde\psi_r])/2$.
Then $\norm{\tilde\omega}\le2\sqrt{2\delta}\norm{\tilde\phi}$ and
$$\begin{aligned}
\betr{\tilde\phi\J xx{a_1-a_2}}&=\frac12\betr{\tilde\psi_l(xx^*)+\tilde\psi_r(x^*x)}\le[\tilde\psi_l](xx^*)+[\tilde\psi_r](x^*x)\\
&\le([\tilde\psi_l]+[\tilde\psi_r])(xx^*+x^*x)=4\tilde\omega(\J xx1)=4\norm{x}_{\tilde\omega}^2.
\end{aligned}$$
It remains to set $\omega=\tilde\omega|_M$.
\end{proof}

\begin{lemma}\label{l ad-hoc 2}
Let $M$ be a JW$^*$-algebra, $\phi\in M_*$ and let $a$ be a norm-one element of $M$.
Then there is a positive functional $\omega\in M_*$ such that
$$\norm{\omega}\le\norm{\phi}\quad\mbox{ and }\quad\forall x\in W:\betr{\phi\J xxa}\le4\norm{x}_\omega^2.$$
\end{lemma}

\begin{proof}
The proof resembles the preceding one of Lemma \ref{l ad-hoc 1 Jordan}. Assume that $M$ is a JW$^*$-subalgebra of a von Neumann algebra $W$ and $1_W\in M$. Let $\tilde\phi\in W_*$ be a norm-preserving extension of $\phi$ (see \cite[Theorem]{Bun01}). Set $\tilde\psi_l=a\tilde\phi$ and $\tilde\psi_r=\tilde\phi a$.
Then $\norm{\tilde\psi_l}\le\norm{a}\norm{\tilde\phi}=\norm{\tilde\phi}$ and, similarly, $\norm{\tilde\psi_r}\le\norm{\tilde\phi}$.
Set $\tilde\omega=([\tilde\psi_l]+[\tilde\psi_r])/2$. Then $\norm{\tilde\omega}\le\norm{\tilde\phi}$ and
$$\begin{aligned}
\betr{\tilde\phi\J xxa}&=\frac12\betr{\tilde\psi_l(xx^*)+\tilde\psi_r(x^*x)}\le[\tilde\psi_l](xx^*)+[\tilde\psi_r](x^*x)\\
&\le([\tilde\psi_l]+[\tilde\psi_r])(xx^*+x^*x)=4\tilde\omega(\J xx1)=4\norm{x}_{\tilde\omega}^2.
\end{aligned}$$
Finally, we may set $\omega=\tilde{\omega}|_M$.
\end{proof}

\begin{proof}[Proof of Proposition~\ref{P:key decomposition}]
It follows from \cite[Theorem 7.1]{Finite} that any JBW$^*$-algebra $M$ can be represented by $M_1\oplus^{\ell_\infty} M_2$ where $M_1$ is a finite JBW$^*$-algebra and $M_2$ is a JW$^*$-algebra. 
The validity of Proposition~\ref{P:key decomposition} for finite JBW$^*$-algebras follows immediately from Observation~\ref{obs:finite JBW* algebras}. 
Since the validity of Proposition~\ref{P:key decomposition} is clearly preserved by $\ell_\infty$-sums, it remains to prove it for JW$^*$-algebras.

So, assume that $M$ is a JW$^*$-algebra and $\varphi\in M_*$.
By homogeneity we may assume $\norm{\varphi}=1$. Fix $\varepsilon>0$. Choose $\delta>0$ such that $12\sqrt{2\delta}<\varepsilon$. 
By the Wright-Youngson extension of the Russo-Dye theorem, the convex hull of all unitary elements in $M$ is norm dense in the closed unit ball of $M$ (see \cite[Theorem 2.3]{WrightYoungson77} or \cite[Fact 4.2.39]{Cabrera-Rodriguez-vol1}). We can therefore find a unitary element $w$ such that $\varphi (w) > 1-\delta$. By Lemma~\ref{l norm-one functional close enough to states at a unitary} there exists a norm-one  functional $\psi\in M_*$ satisfying $\psi(w)=1$ 
and $\|\varphi -\psi\| < \sqrt{2\delta}.$
Set $u=s(\varphi)$.\smallskip

For $x\in M$ we then have
$$\begin{aligned}
 \norm{x}_\varphi^2 &= \varphi \J xxu = \psi \J xxw
 +(\varphi-\psi)\J xxw + \varphi \J xx{u-w}.
\end{aligned}$$

Applying Lemma~\ref{l ad-hoc 2} to $\varphi-\psi$ and $w$ we find a positive functional $\omega_1\in M_*$ with $\norm{\omega_1}\le\norm{\varphi-\psi}<\sqrt{2\delta}$ such that
$$\abs{(\varphi-\psi)\J xxw}\le 4\norm{x}_{\omega_1}^2\mbox{ for }x\in M.$$

Applying Lemma~\ref{l ad-hoc 1 Jordan} to the functional $\varphi$ and the pair $w,u\in M$ we get a positive functional $\omega_2\in M_*$ with $\norm{\omega_2}\le2\sqrt{2\delta}$ such that
$$\abs{\varphi \J xx{u-w}}\le 4\norm{x}_{\omega_2}^2\mbox{ for }x\in M.$$

Hence we have for each $x\in M$
$$\norm{x}_\varphi^2\le \norm{x}_\psi^2+4(\norm{x}_{\omega_1}^2+\norm{x}_{\omega_2}^2)=\norm{x}_\psi^2+\norm{x}_{4(\omega_1+\omega_2)}^2,$$
where we used that $\omega_1$ and $\omega_2$ are positive functionals. Since $s(\psi)\le w$ (just have in mind that $\psi(w)=1$ and \eqref{eq minimality of the support tripotent}), $w$ is unitary and 
$$\norm{4(\omega_1+\omega_2)}<12\sqrt{2\delta},$$
it is enough to set $\varphi_1=\psi$ and $\varphi_2=4(\omega_1+\omega_2)$.
\end{proof}

\begin{remark}\label{Rem}
(1) Note that by \cite[Proposition 7.5]{Finite} any finite tripotent in a JBW$^*$-algebra is majorized by a unitary element, hence Proposition~\ref{P:type I approx} is indeed a stronger version of Proposition~\ref{P:key decomposition} in the special case in which the JBW$^*$-algebra $M$ is a direct sum of a finite JBW$^*$-algebra and a type $I$ JBW$^*$-algebra. (For (\ref{it:key decomposition i}) and (\ref{it:key decomposition iv}) of Proposition~\ref{P:key decomposition} see the remarks before the statement of Proposition~\ref{P:type I approx}.)
Further, as will be seen at the beginning of the next section, Proposition~\ref{P:key decomposition} is the main ingredient for proving Proposition~\ref{P:key decomposition alternative}.

(2) There is an alternative way of proving Proposition~\ref{P:key decomposition}.  It follows from \cite[Theorem 7.1]{Finite} that any JBW$^*$-algebra $M$ can be represented by $M_1\oplus^{\ell_\infty} M_2\oplus^{\ell^\infty} M_3$ where $M_1$ is a finite JBW$^*$-algebra, $M_2$ is a type I JBW$^*$-algebra and $M_3$ is a von Neumann algebra. So, we can conclude using Proposition~\ref{P:type I approx} and giving the above argument only for von Neumann algebras (which is slightly easier).
\end{remark}

\section{Proofs of the main results}\label{sec:proofs}

We start by proving Proposition~\ref{P:key decomposition alternative}.

\begin{proof}[Proof of Proposition~\ref{P:key decomposition alternative}.]
Let $M$ be a JBW$^*$-algebra, $\varphi\in M_*$ and $\varepsilon>0$. By homogeneity we may assume that $\norm{\varphi}=1$. Let $\varphi_1,\varphi_2$ and $w$ correspond to $\varphi$ and $\frac\varepsilon2$ by Proposition~\ref{P:key decomposition}. 
Since $w$ is unitary, we have $M_2(w)=M$, hence we may apply Lemma~\ref{L:rotation} to get $\psi_2\in M_*$ such that
$$s(\psi_{2})\le w, \ \norm{\psi_{2}}\le\norm{\varphi_{2}},\ \norm{\cdot}_{\varphi_{2
}}\le\sqrt{2}\norm{\cdot}_{\psi_{2}}.
$$
Then
$$\begin{aligned}
 \norm{\cdot}_\varphi^2&\le\norm{\cdot}_{\varphi_1}^2+\norm{\cdot}_{\varphi_2}^2\le
  \norm{\cdot}_{\varphi_{1}}^2+2\norm{\cdot}_{\psi_2}^2=\norm{\cdot}_{\varphi_{1}+2\psi_2}^2
 =(\norm{\varphi_{1}}+2\norm{\psi_2})\norm{\cdot}_\psi^2,
 \end{aligned}
$$
where $$\psi=\frac{\varphi_{1}+2\psi_2}{\norm{\varphi_{1}}+2\norm{\psi_2}}.$$
(Note that the first equality follows from the fact that the support tripotents of both functionals are below $w$.)
Since  the functionals $\varphi_{1}$ and $\psi_2$ attain their norms at $w$, we deduce that $\norm{\psi}=1$. It remains to observe
that
$$\norm{\varphi_{1}}+2\norm{\psi_2}\le \norm{\varphi}+2\norm{\varphi_2}\le1+\varepsilon.$$
This completes the proof.
\end{proof}

Having proved Proposition~\ref{P:key decomposition alternative}, we know that Proposition~\ref{P:majorize 1+2+epsilon} is valid as well. Using it and Theorem~\ref{T:triples-dual} we get the following theorem.

\begin{thm}\label{T:JBW*-algebras}
Let $M$ be a JBW$^*$-algebra, let $H$ be a Hilbert space and let $T:M\to H$ be a weak$^*$-to-weak continuous linear operator. Given $\varepsilon>0$, there is a norm-one functional $\varphi\in M_*$ such that
$$\norm{Tx}\le(\sqrt2+\varepsilon)\norm{T}\norm{x}_\varphi\mbox{ for }x\in M.$$
\end{thm}

Now we get the main result by the standard dualization.

\begin{proof}[Proof of Theorem~\ref{t constant >sqrt2 in LG for JBstar algebras}]
Let $T:B\to H$ be a bounded linear operator from a JB$^*$-algebra into a Hilbert space. Let $\varepsilon>0$. Since Hilbert spaces are reflexive, the second adjoint operator $T^{**}$ maps $B^{**}$ into $H$ and it is weak$^*$-to-weak continuous. Further, $B^{**}$ is a JBW$^*$-algebra (cf. \cite[Theorem 4.4.3]{hanche1984jordan} and \cite{Wright1977} or \cite[Proposition 5.7.10]{Cabrera-Rodriguez-vol2} and \cite[Theorems 4.1.45 and 4.1.55]{Cabrera-Rodriguez-vol1}), so Theorem~\ref{T:JBW*-algebras} provides the respective functional $\varphi\in (B^{**})_*=B^*$.
\end{proof}

We further note that for JB$^*$-algebras we have two different forms of the Little Grothendieck theorem -- a triple version (the just proved Theorem~\ref{t constant >sqrt2 in LG for JBstar algebras}) and an algebraic version (an analogue of Theorem~\ref{T:C*alg-sym}). The difference is that the first form provides just a norm-one functional while the second one provides a state, i.e., a positive norm-one functional. Let us now show that the algebraic version may be proved from the triple version.

\begin{thm}\label{T:algebraic version dual}
Let $M$ be a JBW$^*$-algebra, let $H$ be a Hilbert space and let $T:M\to H$ be a weak$^*$-to-weak continuous linear operator. Given $\varepsilon>0$, there is a state $\varphi\in M_*$ such that
$$\norm{Tx}\le(2+\varepsilon)\norm{T}\varphi(x\circ x^*)^{1/2}\mbox{ for }x\in M.$$
\end{thm}

\begin{proof}
By Theorem~\ref{T:JBW*-algebras} there is a norm-one functional $\psi\in M_*$ such that
$$\norm{Tx}\le (\sqrt2+\frac{\varepsilon}{\sqrt2})\norm{T}\norm{x}_\psi\mbox{ for }x\in M.$$
Since $M$ is unital and $M_2(1)=M$, Lemma~\ref{L:rotation} yields a norm-one functional $\varphi\in M_*$ with $s(\varphi)\le1$ and $\norm{\cdot}_\psi\le\sqrt2\norm{\cdot}_\varphi$. Then
$\varphi$ is a state (note that $\varphi(1)=1$) and
$$\norm{Tx}\le(2+\varepsilon)\norm{T}\norm{x}_\varphi\mbox{ for }x\in M.$$
It remains to observe that 
$$\norm{x}_\varphi=\sqrt{\varphi\J xx1}=\sqrt{\varphi(x\circ x^*)}$$
for $x\in M$.
\end{proof}

\begin{thm}\label{T:algebraic version non-dual}
Let $B$ be a JB$^*$-algebra, let $H$ be a Hilbert space and let $T:B\to H$ be a bounded linear operator.  Then there is a state $\varphi\in B^*$ such that
$$\norm{Tx}\le2\norm{T}\varphi(x\circ x^*)^{1/2}\mbox{ for }x\in B.$$
\end{thm}

\begin{proof}
Since $B^{**}$ is a JBW$^*$-algebra, $T^{**}$ maps $B^{**}$ into $H$ and $T^{**}$ is weak$^*$-to-weak continuous, by Theorem~\ref{T:algebraic version dual} we get a sequence $(\varphi_n)$ of states on $B$ such that
$$\norm{Tx}\le(2+\frac1n)\norm{T}\varphi_n(x\circ x^*)^{1/2}\mbox{ for }x\in B\mbox{ and } n\in\en.$$
Let $\tilde\varphi$ be a weak$^*$-cluster point of the sequence $(\varphi_n)$. Then $\tilde\varphi$ is positive, $\norm{\tilde\varphi}\le 1$
and 
$$\norm{Tx}\le2\norm{T}\tilde\varphi(x\circ x^*)^{1/2}\mbox{ for }x\in B.$$
Now we can clearly replace $\tilde\varphi$ by a state. Indeed, if $\tilde\varphi\ne0$, we take $\varphi=\frac{\tilde\varphi}{\norm{\tilde\varphi}}$. If $\tilde\varphi=0$, then $T=0$ and hence $\varphi$ may be any state. (Note that in case $B$ is unital, $\tilde\varphi$ is already a state.)
\end{proof}

We finish this section by showing that our main result easily implies Theorem~\ref{T-C*alg}.

\begin{proof}[Proof of Theorem~\ref{T-C*alg} from Theorem~\ref{t constant >sqrt2 in LG for JBstar algebras}] Let $A$ be a C$^*$-algebra, let $H$ be a Hilbert space and let $T:A\to H$ be a bounded linear operator. By Theorem~\ref{t constant >sqrt2 in LG for JBstar algebras} there is a sequence $(\psi_n)$ of norm-one functionals in $A^*$ such that
$$\norm{Tx}\le (\sqrt{2}+\frac1n)\norm{T}\norm{x}_{\psi_n}\mbox{ for }x\in A \mbox{ and }n\in\en.$$
Recall that $A^{**}$ is a von Neumann algebra. Set $u_n=s(\psi_n)\in A^{**}$. Then 
$$\norm{x}_{\psi_n}^2=\psi_n\J xx{u_n}=\frac12(\psi_n(xx^*u_n)+\psi_n(u_nx^*x))=\frac12(u_n\psi_n(xx^*)+\psi_n u_n(x^*x))
$$
for $x\in A$. Moreover,  $\varphi_{1,n}=u_n\psi_n$ and
$\varphi_{2,n}=\psi_n u_n$
are states on $A$ (note that $\varphi_{1,n}=\abs{\psi_n}$ and $\varphi_{2,n}=\abs{\psi_n^*}$) such that
$${\norm{Tx}\le (\sqrt{2}+\frac1n)\norm{T}\cdot\frac1{\sqrt{2}}(\varphi_{1,n}(xx^*)+\varphi_{2,n}(x^*x))^{1/2}\mbox{ for }x\in A \mbox{ and }n\in\en.}$$
Let $(\varphi_1,\varphi_2)$ be a weak$^*$-cluster point of the sequence $((\varphi_{1,n},\varphi_{2,n}))_n$ in $B_{A^*}\times B_{A^*}$.
Then $\varphi_1,\varphi_2$ are positive functionals of norm at most one such that
$$\norm{Tx}\le \|T\| (\varphi_{1}(xx^*)+\varphi_{2}(x^*x))^{1/2}\mbox{ for }x\in A.$$
Similarly as above we may replace $\varphi_1$ and $\varphi_2$ by states.
\end{proof}

\section{Examples and problems}\label{sec:problems}

\begin{ques}
Do Theorem~\ref{t constant >sqrt2 in LG for JBstar algebras} and Theorem~\ref{T:JBW*-algebras} hold with the constant $\sqrt2$ instead of $\sqrt2+\varepsilon$?
\end{ques}

We remark that these theorems do not hold with a constant strictly smaller than $\sqrt{2}$. Indeed, assume that Theorem~\ref{t constant >sqrt2 in LG for JBstar algebras} holds with a constant $K$. Then Theorem~\ref{T-C*alg} holds with constant $\frac{K}{\sqrt{2}}$ (see the proof of the relationship of these two theorems in Section~\ref{sec:proofs}). But the best constant for Theorem~\ref{T-C*alg} is $1$ due to \cite{haagerup-itoh}.

Since the example in \cite{haagerup-itoh} uses a rather involved combinatorial construction, we provide an easier example showing that the constant in Theorem~\ref{t constant >sqrt2 in LG for JBstar algebras} has to be at least $\sqrt{2}$.

\begin{example2}\label{ex:Tx=xxi} Let $H$ be an infinite-dimensional Hilbert space. Let $A=K(H)$ be the C$^*$-algebra of compact operators.
Fix an arbitrary unit vector $\xi\in H$ and define $T:A\to H$ by $Tx=x\xi$ for $x\in A$. It is clear that $\norm{T}=\norm{\xi}=1$. Fix an arbitrary norm-one functional $\varphi\in A^*$. We are going to prove that
\begin{equation}
\sup \left\{\frac{\norm{Tx}}{\norm{T} \norm{x}_\varphi}\setsep x\in A, \norm{x}_\varphi\neq0\right\}\ge\sqrt{2}.\label{eq Example} 
\end{equation}

Recall that $K(H)^*$ is identified with $N(H)$, the space of nuclear operators on $H$ equipped with the nuclear norm, and $K(H)^{**}$ is identified with $B(H)$, the von Neumann algebra of all bounded linear operators on $H$.
Using the trace duality we deduce that there is a nuclear operator $z$ on $H$ such that $\tr{\abs{z}}=\norm{z}_N=1$ and $\varphi(x)=\tr{zx}$ for $x\in A$. Consider the polar decomposition $z=u\abs{z}$ in $B(H)$. Then $\abs{z}=u^*z$, hence $s(\varphi)\le u^*$. (Note that $\varphi(u^*)=\tr{zu^*}=\tr{u^*z}=\tr{\abs{z}}=1$, hence
$s(\varphi)\le u^*$ by \eqref{eq minimality of the support tripotent}. The converse inequality holds as well, but it is not important.)
It follows that for each $x\in A$ we have
$$\begin{aligned}\norm{x}_\varphi^2&=\varphi(\J xx{u^*})=\frac12\varphi(xx^*u^*+u^*x^*x)=
\frac12\tr{xx^*u^*z+u^*x^*xz}\\&=\frac12(\tr{xx^*\abs{z}}+\tr{u^*x^*xz})\end{aligned}$$
If $\eta\in H$ is a unit vector, we define the operator
$$y_\eta(\zeta)=\ip{\zeta}{\xi}\eta,\qquad\zeta\in H.$$
Then $y_\eta\in A$, $\norm{y_\eta}=1$ and $\norm{Ty_\eta}=1$. Moreover,
$$y_\eta^*(\zeta)=\ip{\zeta}{\eta}\xi,$$
hence
$$y_\eta y_\eta^*(\zeta)=\ip{\zeta}{\eta}\eta\mbox{ and }y_\eta^*y_\eta(\zeta)=\ip{\zeta}{\xi}\xi.$$
Thus
$$\norm{y_\eta}_\varphi^2=\frac12(\tr{\abs{z}y_\eta y_\eta^*}+\tr{zu^* y_\eta^*y_\eta})=\frac12(\ip{\abs{z}\eta}{\eta}+\ip{zu^*\xi}{\xi})\le\frac12(1+\ip{\abs{z}\eta}{\eta}).$$
It follows that
$$\inf\{\norm{x}_\varphi^2\setsep x\in A, \norm{Tx}=1\}\le \frac 12 \inf\{1+\ip{\abs{z}\eta}{\eta}\setsep \norm{\eta}=1\}=\frac12+\frac12\min \sigma(\abs{z}),$$
where the last equality follows from \cite[Theorem 15.35]{fabianetal2011}.
Now, $z$ is a nuclear operator of norm one. Thus $0\in\sigma(\abs{z})$ as $H$ has infinite dimension.
Hence
$$\inf\{\norm{x}_\varphi\setsep x\in A,\norm{Tx}=1\}\le\frac1{\sqrt2},$$
which yields inequality \eqref{eq Example}.
\qed\end{example2}

\begin{remark}
If $H$ is a finite-dimensional Hilbert space, the construction from Example~\ref{ex:Tx=xxi} could be done as well. In this case $A=K(H)=B(H)$ can be identified with the algebra of $n\times n$ matrices where $n=\dim H$. In this case $\sigma(\abs{z})$ need not contain $0$, but at least one of the eigenvalues of $\abs{z}$ is at most $\frac1n$. So, we get a lower bound $\sqrt{\frac{2n}{n+2}}$ for the constant in Theorem~\ref{t constant >sqrt2 in LG for JBstar algebras}.
\end{remark}

Next we address the optimality of the algebraic version of the Little Grothendieck theorem. 

\begin{ques}
What is the optimal constant in Theorem~\ref{T:C*alg-sym},
Theorem~\ref{T:algebraic version dual} and Theorem~\ref{T:algebraic version non-dual}? In particular,
do these theorems hold with the constant $\sqrt{2}$?
\end{ques}

Note that the constant cannot be smaller than $\sqrt2$ due to Example~\ref{ex:Tx=xxi}. The following example shows that Example~\ref{ex:Tx=xxi} cannot yield a greater lower bound.

\begin{example2}
Let $H$, $A$, $\xi$ and $T$ be as in Example~\ref{ex:Tx=xxi}. Let $u\in A^{**}=B(H)$ be any unitary element. Then
$$\varphi_u(x)=\ip{x\xi}{u\xi},\qquad x\in A$$
defines a norm-one functional in $A^*$ such that $s(\varphi_u)\le u$ and, moreover,
$$\norm{Tx}\le\sqrt2 \norm{x}_{\varphi_u}\mbox{ for }x\in A.$$

Indeed, it is clear that $\norm{\varphi_u}\le 1$. Since $\varphi_u(u)=1$, necessarily $\norm{\varphi_u}=1$ and $s(\varphi)\le u$. Moreover, for $x\in A$ we have
$$\begin{aligned}
\norm{x}_{\varphi_u}^2&=\varphi_u\J xxu=\frac12\varphi_u(xx^*u+ux^*x)
=\frac12(\ip{xx^*u\xi}{u\xi}+\ip{ux^*x\xi}{u\xi})
\\&=\frac12(\norm{x^*u\xi}^2+\norm{x\xi}^2)\ge\frac12\norm{x\xi}^2=\frac12\norm{Tx}^2.\end{aligned}$$
This completes the proof.
\qed\end{example2}
 
We continue by recalling the example of \cite{haagerup-itoh} showing optimality of Theorem~\ref{T-C*alg} and explaining that it does not show optimality neither of Theorem~\ref{T:C*alg-sym} nor of Theorem~\ref{T:algebraic version non-dual}.

An important tool to investigate optimality of constants in Theorem~\ref{T-C*alg} is the following characterization.

\begin{prop}[{\cite[Proposition 23.5]{pisier2012grothendieck}}]\label{p equivalent formulation C*}
Let $A$ be a C$^*$-algebra, $H$ a Hilbert space, $T:A\rightarrow H$ a bounded linear map and $K$ a positive number.
Then the following two assertions are equivalent.
\begin{enumerate}[(i)]
\item\label{it equiv formul2 C*} There are states
$\varphi_1, \varphi_2$ on $A$ such that
\begin{eqnarray}
 \norm{Tx}\le K\norm{T}(\varphi_1(x^*x)+\varphi_2(xx^*))^{1/2} \quad \mbox{for } x\in A.   
\end{eqnarray}
\item\label{it equiv formul1 C*} For any finite sequence $(x_j)$ in $A$ we have
\begin{eqnarray}
\left(\sum_j\norm{Tx_j}^2\right)^{1/2}\le K\norm{T}\left(\Norm{\sum_j x_j^*x_j}+\Norm{\sum_j x_jx_j^*}\right)^{1/2}.
\label{eq equiv formul1 C*}
\end{eqnarray}
\end{enumerate}
\end{prop}

The following proposition is a complete analogue of the preceding one and can be used to study optimality of Theorem~\ref{T:algebraic version non-dual}. We have not found it explicitly formulated in the literature, but its proof is completely analogous
to the proof of Proposition~\ref{p equivalent formulation C*} given in \cite{pisier2012grothendieck}.

 \begin{prop}\label{p equivalent formulation JB*-algebra}
Let $A$ be a unital JB$^*$-algebra, $H$ a Hilbert space, $T:A\rightarrow H$ a bounded linear map and $K$ a positive number.
Then the following two assertions are equivalent.
\begin{enumerate}[(i)]
\item\label{it equiv formul2 JB*} There is a state
$\varphi$ on $A$ such that
\begin{eqnarray}
 \norm{Tx}\le K\norm{T}\varphi(x^*\circ x)^{1/2} \quad \mbox{for } x\in A.
\end{eqnarray}
\item\label{it equiv formul1 JB*} For any finite sequence $(x_j)$ in $A$ we have
\begin{eqnarray}
\left(\sum_j\norm{Tx_j}^2\right)^{1/2}\le K\norm{T}\Norm{\sum_j x_j^*\circ x_j}^{1/2}.
\label{eq equiv formul1 JB*}
\end{eqnarray}
\end{enumerate}
\end{prop}

We recall the example originated in \cite{haagerup-itoh} and formulated and proved in this setting in \cite{pisier2012grothendieck}.

\begin{example}[{\cite[Lemma 11.2]{pisier2012grothendieck}}]\label{example}
Consider an integer $n\ge1$. Let $N=2n+1$ and $d=\begin{pmatrix}2n+1\\n\end{pmatrix}=\begin{pmatrix}2n+1\\n+1\end{pmatrix}$.
Let $\tau_d$ denote the normalized trace on  the space $M_d$ of $d\times d$ (complex) matrices. There are $x_1, \ldots, x_N$
in $M_d$ such that $\tau_d(x_i^*x_j)=1$ if $i=j$ and $=0$ otherwise, satisfying
\begin{eqnarray}
\sum_j x_j^*x_j=\sum_j x_jx_j^*=NI  \label{eq1 example}
\end{eqnarray}
and moreover such that, with $a_n=(n+1)/(2n+1)$,
\begin{eqnarray}
\forall \alpha=(\alpha_i)\in\ce^N,\quad  \Norm{\sum
_j\alpha_jx_j}_{(M_d)^*}=d\sqrt{a_n}\left(\sum_j\betr{\alpha_j}^2\right)^{1/2}.
\label{eq2 example}
\end{eqnarray}
\end{example}

In the following example we show that the previous one yields the optimality of Theorem~\ref{T-C*alg} but does not help to find the optimal constant for Theorem~\ref{T:C*alg-sym} or Theorem~\ref{T:algebraic version non-dual}. The first part is proved already in \cite{haagerup-itoh} (cf. \cite[Section 11]{pisier2012grothendieck}) but we include the proof for the sake of completeness and, further, in order to compare it with the second part.

\begin{example2} Fix $n\ge1$.
With the notation of Example \ref{example}
define $T:M_d\to \ell_2^N$ by $$T(x)=(\tau_d(x_j^*x))_{j=1}^N,\qquad x\in M_d.$$ 
Let $(\eta_j)_{j=1}^N$ be the canonical orthonormal basis of $\ell_2^N$. Then the dual mapping $T^*:\ell_2^N\to M_d^*$ fulfils
$$\ip{T^*(\eta_j)}{x}=\ip{\eta_j}{T(x)}=\tau_d(x_j^*x)=\frac1d\tr{x_j^*x}\mbox{ for }x\in M_d,$$
thus $T^*(\eta_j)=\frac1dx_j^*$ (we use the trace duality). Then \eqref{eq2 example} shows that
$$\norm{T^*(\alpha)}=\frac1d\norm{\sum_{j=1}^N\alpha_jx_j^*}_{(M_d)^*}=\sqrt{a_n}\norm{\alpha}\mbox{ for }\alpha\in\ell_2^N.$$
In particular, $\frac1{\sqrt{a_n}}T^*$ is an isometric embedding, thus $\frac1{\sqrt{a_n}}T$ is a quotient mapping. Hence, $\norm{T}=\sqrt{a_n}$.

Further, $T(x_j)=\eta_j$ for $j=1,\dots,N$, so
$$\sum_{j=1}^N \norm{T(x_j)}^2=N$$
and
$$\norm{\sum_{j=1}^N x_j^*x_j}+\norm{\sum_{j=1}^N x_jx_j^*}=2\norm{NI}=2N.$$
Thus due to Proposition~\ref{p equivalent formulation C*} the optimal value of the constant in Theorem~\ref{T-C*alg} is bounded below by
$$\frac{1}{\sqrt{2a_n}}=\sqrt{\frac{2n+1}{2n+2}}\to 1.$$

On the other hand,
$$\norm{\sum_{j=1}^N x_j^*\circ x_j}=\norm{NI}=N,$$
thus Proposition~\ref{p equivalent formulation JB*-algebra} yields that the optimal value of the constant in Theorem~\ref{T:C*alg-sym} is bounded below by
$$\frac1{\sqrt{a_n}}=\sqrt{\frac{2n+1}{n+1}}\to \sqrt2,$$
so it gives nothing better than Example~\ref{ex:Tx=xxi}.

In fact, this operator $T$ satisfies Theorem~\ref{T:C*alg-sym} with constant $\frac1{\sqrt{a_n}}\le\sqrt{2}$. 

To see this observe that $(x_j)_{j=1}^N$ is an orthonormal system in $M_d$ equipped with the normalized Hilbert-Schmidt inner product.
Hence, any $x\in M_d$ can be expressed as
$$x=y+\sum_{j=1}^N\alpha_jx_j,$$
where $\alpha_j$ are scalars and $y\in\{x_1,\dots,x_N\}^{\perp_{HS}}$.
Then $T(x)=(\alpha_j)_{j=1}^N$ and
$$\tau_d(x^*\circ x)=\tau_d(x^*x)=\tau_d(y^*y)+\sum_{j=1}^N\abs{\alpha_j}^2\ge \sum_{j=1}^N\abs{\alpha_j}^2=\norm{T(x)}^2.$$
Hence
$$\norm{T(x)}\le \tau_d(x^*\circ x)^{1/2}=\frac1{\sqrt{a_n}}\norm{T} \tau_d(x^*\circ x)^{1/2}.$$
Since $\tau_d$ is a state, the proof is complete.
\end{example2}

We continue by an example showing that there is a real difference between the triple and algebraic versions of the Little Grothendieck theorem.

\begin{example}\label{ex:alg vs triple} \
\begin{enumerate}[$(a)$]
    \item Let $M$ be any JBW$^*$-triple and let 
$\varphi\in M_*$ be a norm-one functional. Then $$\abs{\varphi(x)}\le\norm{x}_\varphi,\mbox{ for all }x\in M,$$
hence $\varphi:M\to\ce$ satisfies Theorem~\ref{T:triples}$(3)$ with constant one.
\item Let $M_2$ be the algebra of $2\times 2$ matrices. Then there is a norm-one functional $\varphi:M_2\to \ce$ not satisfying Theorem~\ref{T:C*alg-sym} with constant smaller than $\sqrt{2}$.
\item In particular, the constant $\sqrt{2}$ in Lemma~\ref{L:rotation} is optimal.
\end{enumerate}
\end{example}

\begin{proof}
$(a)$ The desired inequality was already stated in \cite[comments before Definition 3.1]{barton1990bounded}. Let us give some details. We set $e=s(\varphi)$. Then
$$\abs{\varphi(x)}=\abs{\varphi(P_2(e)x)}=\abs{\varphi(\J{P_2(e)x}ee}\le\norm{P_2(e)x}_\varphi\norm{e}_\varphi=\norm{P_2(e)x}_\varphi.$$
Moreover,
$$\begin{aligned}\norm{x}_\varphi^2&=\varphi\J xxe=\varphi(P_2(e)\J xxe)\\&=\varphi(\J{P_2(e)x}{P_2(e)x}e+\J{P_1(e)x}{P_1(e)x}e)=\norm{P_2(e)x}_\varphi^2+\norm{P_1(e)x}_\varphi^2\\&\ge\norm{P_2(e)x}_\varphi^2.\end{aligned}$$

$(b)$ Each $a\in M_2$ can be represented as $a=(a_{ij})_{i,j=1,2}$. Define $\varphi:M_2\to\ce$ by
$$\varphi(a)=a_{12},\quad a\in M_2.$$
{
It is clear that $\norm{\varphi}=1$ and that $\varphi(s)=1$ where
$$s=\begin{pmatrix} 0 & 1\\ 0& 0\end{pmatrix}.$$
Let $\psi$ be any state on $M_2$.
Then
$$\norm{s}_\psi^2=\psi (\J ss{\mathbf{1}})=\frac12\psi(s s^*+s^*s)=\frac12\psi(\mathbf{1})=\frac12.$$
Thus $\varphi(s)=\sqrt{2}\norm{s}_\psi$ for any state $\psi$ on $A=M_2$, which completes the proof.

$(c)$ This follows from $(b)$ (consider $p=\mathbf{1}$).}
\end{proof}

\section{Notes and problems on general JB$^*$-triples}\label{sec:triples}

The main result, Theorem~\ref{t constant >sqrt2 in LG for JBstar algebras}, is formulated and proved for JB$^*$-algebras. The assumption that we deal with a JB$^*$-algebra, not with a general JB$^*$-triple, was strongly used in the proof. Indeed, the key step was to prove
the dual version for JBW$^*$-algebras, Theorem~\ref{T:JBW*-algebras},
and we substantially used the existence of unitary elements. So, the following problem remains open.

\begin{ques}
Is Theorem~\ref{t constant >sqrt2 in LG for JBstar algebras} valid for general JB$^*$-triples?
\end{ques}

We do not know how to attack this question. However, there are some easy partial results. Moreover, some of our achievements may be easily extended to JBW$^*$-triples. In this section we collect such results.

The first example shows that for some JB$^*$-triples the optimal constant in the Little Grothendieck Theorem is easily seen to be $\sqrt2$. This is shown by completely elementary methods.

\begin{example2}
Let $H$ be a Hilbert space considered as the triple $B(\ce,H)$ (i.e., a type 1 Cartan factor). That is, the triple product is given by
$$\J xyz=\frac12(\ip xyz+\ip zyx),\quad x,y,z\in H.$$
The dual coincides with the predual and it is isometric to $H$. Let $y\in H^*$ be a norm-one element, i.e. we consider it as the functional $\ip{\cdot}{y}$. Then 
$s(y)=y$. So, for $x\in H$ we have
$$\norm{x}_y^2=\ip{\J xxy}{y}=
\frac12\ip{\ip xxy+\ip yxx}{y}=\frac12(\norm{x}^2+\abs{\ip xy}^2)\ge\frac12\norm{x}^2.$$
Hence, if $K$ is another Hilbert space and $T:H\to K$ is a bounded linear operator, then for any norm-one $y\in H^*$ we have
$$\norm{Tx}\le\norm{T}\norm{x}\le\sqrt{2}\norm{T}\norm{x}_y,$$
so we have the Little Grothendieck theorem with constant $\sqrt2$.\smallskip

Moreover, the constant $\sqrt2$ is optimal in this case as soon as $\dim H\ge2$. Indeed, let $T:H\to H$ be the identity. Given any norm-one element $y\in H$, we may find a norm-one element $x\in H$ with $x\perp y$. The above computation shows that $\norm{x}=\sqrt{2}\norm{x}_y$.
\end{example2}

Another case, nontrivial but well known, is covered by the following example.

\begin{example2}
Assume that $E$ is a finite-dimensional JB$^*$-triple. Then $E$ is reflexive and, moreover, any bounded linear operator $T:E\to H$ (where $H$ is a Hilbert space) attains its norm. Hence $E$ satisfies the Little Grothendieck theorem with constant $\sqrt{2}$ by Theorem~\ref{T:triples}$(1)$.
\end{example2}

We continue by checking which methods used in the present paper easily work for general triples.

\begin{obs}\label{obs:type I approx triples}
Proposition~\ref{P:type I approx} holds for corresponding JBW$^*$-triples as well.
\end{obs}

\begin{proof}
It is clear that it is enough to prove it separately for finite JBW$^*$-triples and for type I JBW$^*$-triples. The case of finite JBW$^*$-triples is trivial (one can take $\varphi_2=0$). 
So, let $M$ be a JBW$^*$-triple of type I, $\varphi\in M_*$ and $\varepsilon>0$. Set $e=s(\varphi)$. Then $M_2(e)$ is a type I JBW$^*$-algebra (see \cite[comments on pages 61-62 or Theorem 4.2]{BuPe02}) and $\varphi|_{M_2(e)}\in M_2(e)_*$. Apply Proposition~\ref{P:type I approx} to $M_2(e)$ and $\varphi|_{M_2(e)}$ to get $\varphi_1$ and $\varphi_2$.
The pair of functionals $\varphi_1\circ P_2(e)$ and $\varphi_2\circ P_2(e)$ completes the proof.
\end{proof}

Observe that the validity of Proposition~\ref{P:type I approx} for finite JBW$^*$-triples is trivial but useless if we have no unitary element. However, the `type I part' may be used at least in some cases.

\begin{prop}\label{P:B(H,K)}
Let $M=L^\infty(\mu)\overline{\otimes}B(H,K)$, where $H$ and $K$ are infinite-di\-men\-sional Hilbert spaces. Then Proposition~\ref{P:majorize 1+2+epsilon} holds for $M$.
\end{prop}

\begin{proof} Let us start by showing that  Peirce-2 subspaces of tripotents in $M$ are upwards directed by inclusion. To this end first observe that $M=pV$, where $V$ is a von Neumann algebra and $p\in V$ is a properly infinite projection.
This is explained for example in \cite[p. 43]{hamhalter2019mwnc}. Now assume that $u_1,u_2\in pV$ are two tripotents (i.e., partial isometries in $V$ with final projections below $p$). By \cite[Lemma 9.8(c)]{hamhalter2019mwnc} there are projections $q_1,q_2\in V$ such that $q_j\ge p_i(u_j)$ and $q_j\sim p$ for $j=1,2$. Further, by \cite[Lemma 9.8(a)]{hamhalter2019mwnc} we have $q_1\vee q_2\sim p$, so there is a partial isometry $u\in V$ with $p_i(u)=q_1\vee q_2$ and $p_f(u)=p$. Then $u\in pV=M$ and $M_2(u)\supset M_2(u_1)\cup M_2(u_2)$.
 
Now we proceed with the proof of the statement itself. Let $\varphi_1,\varphi_2\in M_*$ and $\varepsilon>0$. Note that $M$ is of type I, hence we may apply Observation~\ref{obs:type I approx triples} to get the respective decomposition $\varphi_1=\varphi_{11}+\varphi_{12}$. 
Let $u\in M$ be a tripotent such that $M_2(u)$ contains $s(\varphi_{11}),s(\varphi_{12}),s(\varphi_2)$. Such a $u$ exists as Peirce-2 subspaces of tripotents in $M$ are upwards directed by inclusion as explained above. We can find a unitary $v\in M_2(u)$ with $s(\varphi_{11})\le v$ (recall that $s(\varphi_{11})$ is a finite tripotent and use \cite[Proposition  7.5]{Finite}). We conclude by applying Lemma~\ref{L:rotation}. 
\end{proof}

Combining the previous proposition with Theorem~\ref{T:triples-dual} we get the following.

\begin{cor}
Let $M=L^\infty(\mu)\overline{\otimes}B(H,K)$, where $H$ and $K$ are infinite-di\-men\-sio\-nal Hilbert spaces. Then Theorem~\ref{T:JBW*-algebras} holds for $M$.
\end{cor}

We finish by pointing out main problems concerning JBW$^*$-triples.

\begin{ques}
Assume that $M$ is a JBW$^*$-triple of one of the following forms:
\begin{itemize}
    \item $M=L^\infty(\mu,C)$, where $\mu$ is a probability measure and $C$ is a finite-di\-men\-sio\-nal JB$^*$-triple without unitary element.
    \item $M=pV$, where $V$ is a von Neumann algebra and $p$ is a purely infinite projection.
    \item $M=pV$, where $V$ is a von Neumann algebra and $p$ is a finite projection.
\end{itemize}
Is Theorem~\ref{T:JBW*-algebras} valid for $M$?
\end{ques}

Note that these three cases correspond to the three cases distinguished in \cite{HKPP-BF}. We conjecture that the second case may be proved by adapting the results of Section~\ref{sec:JW*} (but we do not see an easy way) and that the third case is the most difficult one (similarly as in \cite{HKPP-BF}).

\begin{remark}{\rm Haagerup applied in \cite{haagerup1985grothendieck} ultrapower techniques to relax some of the extra hypotheses assumed by Pisier in the first approach to a Grothendieck inequality for C$^*$-algebras. We should include a few words justifying that Haagerup's techniques are not effective in the setting of JB$^*$-triples. Indeed, while
a cluster point (in a reasonable sense) of states of a unital C$^*$-algebra is a state, a cluster
point of norm-one functionals may be even zero. It is true for weak (weak$^*$) limits and also for ultrapowers. The ultrapower, $E_{\mathcal{U}},$ of a JB$^*$-triple, $E$, with respect to an ultrafilter $\mathcal{U}$, is again a JB$^*$-triple with respect to the natural extension of the triple product (see \cite[Corollary 10]{Dineen86}), and $E$ can be regarded as a JB$^*$-subtriple of $E_{\mathcal{U}}$ via the inclusion of elements as constant sequences. Given a norm one functional $\widetilde{\varphi}\in  E_{\mathcal{U}}^*$ the restriction $\varphi = \widetilde{\varphi}|_{E}$ belongs to $E^*$ however we cannot guarantee that $\|x\|_{\widetilde{\varphi}} = \|[x]_{\mathcal{U}}\|_{\widetilde{\varphi}}$ is bounded by a multiple of $\|x\|_{{\varphi}}$. Let us observe that both prehilbertian seminorms coincide on elements of $E$ when the latter is a unital C$^*$-algebra and $\widetilde{\varphi}$ is a state on $E.$

}\end{remark}

\medskip

\textbf{Acknowledgements}
A.M. Peralta partially supported by the Spanish Ministry of Science, Innovation and Universities (MICINN) and European Regional Development Fund project no. PGC2018-093332-B-I00, the IMAG–Mar{\'i}a de Maeztu grant CEX2020-001105-M/AEI/10.13039/501100011033, and by Junta de Andaluc\'{\i}a grants FQM375 and A-FQM-242-UGR18. \smallskip\smallskip

We would like to thank the referees for their carefully reading of our manuscript and their constructive comments.

\def\cprime{$'$} \def\cprime{$'$}

\end{document}